\documentclass[pdflatex,sn-mathphys-num]{sn-jnl}% Math and Physical Sciences Numbered Reference Style
\usepackage{graphicx}%
\usepackage{multirow}%
\usepackage{amsmath,amssymb,amsfonts}%
\usepackage{amsthm}%
\usepackage{mathrsfs}%
\usepackage[title]{appendix}%
\usepackage{xcolor}%
\usepackage{textcomp}%
\usepackage{manyfoot}%
\usepackage{booktabs}%
\usepackage{algorithm}%
\usepackage{algorithmicx}%
\usepackage{algpseudocode}%
\usepackage{listings}%
\usepackage{enumitem}
\theoremstyle{thmstyleone}%
\newtheorem{Theorem}{Theorem}[section]%  meant for continuous numbers
\newtheorem{Proposition}{Proposition}[section]%
\newtheorem{Corollary}{Corollary}[section]

\newtheorem{Lemma}{Lemma}[section]

\theoremstyle{thmstyletwo}%
\newtheorem{Remark}{\bfseries Remark}[section]

\theoremstyle{thmstylethree}%

\def\theequation{\thesection.\arabic{equation}}
\makeatletter
   \@addtoreset{equation}{section}
\makeatother
\usepackage{geometry}
\begin{document}

%% 开启行号功能
%\linenumbers

\title[]{Global spherically  symmetric solutions to the multidimensional isentropic compressible Navier--Stokes--Korteweg system with large initial data}


\author[1]{\fnm{Zhengzheng } \sur{Chen}}

\author[2]{\fnm{Fanfan} \sur{ Jiang}}

%\affil[1]{\orgdiv{School of Mathematical Sciences}, \orgname{Anhui University}, \orgaddress{\street{}\city{Hefei}, \postcode{230601}, \state{}\country{China}}}

\affil[1, 2]{\orgdiv{School of Mathematical Sciences}, \orgname{Anhui University}, \orgaddress{\street{}\city{Hefei}, \postcode{230601}, \state{}\country{China}}}

%\affil[2]{\orgdiv{Department}, \orgname{Organization}, \orgaddress{\street{Street}, \city{City}, \postcode{10587}, \state{State}, \country{Country}}}
%
%\affil[3]{\orgdiv{Department}, \orgname{Organization}, \orgaddress{\street{Street}, \city{City}, \postcode{610101}, \state{State}, \country{Country}}}


\abstract{\unboldmath In this paper, we investigate the global existence of spherically symmetric strong solutions with large initial data to an initial-boundary value problem of the multidimensional isentropic compressible Navier-Stokes-Korteweg system in an unbounded exterior domain. We consider the case when the pressure $p(\rho)=\rho^\gamma$, the viscosity coefficients $\mu(\rho)$  and $ \lambda(\rho)$ satisfy either $\mu(\rho)=\tilde{\mu}, \lambda(\rho)=\tilde{\lambda}\rho^\alpha$ or $\mu(\rho)=\tilde{\mu}\rho^\alpha, \lambda(\rho)=\tilde{\lambda}\rho^\alpha$,  and the capillarity coefficient $\kappa(\rho)=\tilde{\kappa}\rho^\beta$, where  $\alpha,\beta,\gamma \in \mathbb{R}$ are parameters, and $\tilde{\mu},\tilde{\lambda},\tilde{\kappa}$ are given real constants. Under suitable restrictions on the  parameters $\alpha,\beta$ and $\gamma$, we establish the global existence and uniqueness of spherically symmetric strong solutions. The proof relies on the radically weighted energy method combined with the technique developed by Y.~Kanel'\cite{28}.
}

\pacs[2000 AMS Subject Classifications:]{35Q30, 76N10, 35B40.}

\maketitle

\section{Introduction }
In this article, we are concerned with the global existence of spherically symmetric solutions to the following initial-boundary value problem of the multidimensional isentropic compressible Navier-Stokes-Korteweg system in the unbounded exterior domain $\Omega:= \{\xi\in\mathbb{R}^d: |\xi| \ge a\}$, where $\xi$ is the spatial variable with space dimension $d\geq2$, and $a$ is given positive constant. The compressible Navier-Stokes-Korteweg system modeling the motions of the viscous  compressible fluids with internal capillarity, can be written in the Eulerian coordinates as
\begin{eqnarray}\label{1.1}
\left\{\begin{array}{ll}
          \partial_t \rho + \text{div}(\rho \mathbf{u}) = 0,\\[2mm]
          \partial_t (\rho \mathbf{u}) + \text{div}(\rho \mathbf{u} \otimes \mathbf{u}) + \nabla p(\rho) = \text{div}(\mathbb{T}+\mathbb{K}),
\end{array}\right.
\end{eqnarray}
where the stress tensor $\mathbb{T}$ and the capillary tensor $\mathbb{K}$ are given by
\begin{equation}\label{1.2}
\begin{split}
          \mathbb{T}&=\mu(\rho)\left(\nabla\mathbf{u}+(\nabla\mathbf{u})^T\right)+\lambda(\rho)\text{div}\mathbf{u}\cdot \mathbb{I},\\[1.5mm]
          \text{div}\mathbb{K}&=\nabla\left(\rho \kappa(\rho)\Delta\rho+\frac{\kappa(\rho)+\rho \kappa'(\rho)}{2}|\nabla\rho|^2\right)-\text{div}\left(\kappa(\rho)\nabla\rho\otimes\nabla\rho\right),
\end{split}
\end{equation}
with $(\nabla\mathbf{u})^T$ and $\mathbb{I}$ being  the transpose matrix of $\nabla\mathbf{u}$,  and  the $d\times d$ identity matrix, respectively. Here the primary dependent variables are the density $\rho(t,\xi)>0$, and the velocity field $\mathbf{u}(t,\xi)\in\mathbb{R}_d$.

We shall consider the initial boundary value problem of system (\ref{1.1}) in the region $(0,\infty)\times\Omega$ with the following initial and boundary conditions:
\begin{eqnarray}\label{1.3}
\left\{\begin{array}{ll}
          (\rho,\boldsymbol{u})\big|_{t=0}=(\rho_0,\boldsymbol{u}_0)(\xi), \,\,\,\,&\xi\in\overline{\Omega},\\[2mm]
          (\partial_{\boldsymbol{n}}\rho,\boldsymbol{u})\big|_{\partial\Omega}=0, \,\,\,\,&t\ge0,
\end{array}\right.
\end{eqnarray}
here $\boldsymbol{n}$ denotes the unit outer normal to $\partial\Omega$ , and we suppose that the boundary conditions are compatible with the initial data.

The viscosity coefficients $\mu(\rho), \lambda(\rho)$ and the thermal conductivity coefficient $\kappa(\rho)$ are prescribed through constitutive relations as functions of the density satisfying
\begin{eqnarray}\label{1.4}
\mu(\rho)>0,~~~~~~\kappa(\rho)>0,~~~~~~2\mu(\rho)+d\lambda(\rho)>0.
\end{eqnarray}
Furthermore, we consider in this paper the pressure function $p(\rho)$, the viscosity coefficients $\mu(\rho)$ and $\lambda(\rho)$, and the thermal conductivity coefficient $\kappa(\rho)$ satisfy one of the following two conditions:
\begin{eqnarray}\label{1.5}
&&p(\rho)=\rho^\gamma,\quad \mu(\rho)=\widetilde{\mu},\quad \lambda(\rho)=\widetilde{\lambda}\rho^\alpha,\quad \kappa(\rho)=\widetilde{\kappa}\rho^\beta,\\[2mm]
&&p(\rho)=\rho^\gamma, \quad \mu(\rho)=\widetilde{\mu}\rho^\alpha,\quad \lambda(\rho)=\widetilde{\lambda}\rho^\alpha,\quad \kappa(\rho)=\widetilde{\kappa}\rho^\beta,\label{1.6}
\end{eqnarray}
where and $\gamma\geq1, \alpha, \beta\in\mathbb{R}$ are parameters, and $\widetilde{\mu}, \widetilde{\lambda}, \widetilde{\kappa}$ are constants,
\begin{Remark}
In particular, if the viscosity coefficients $\mu(\rho)$ and $\lambda(\rho)$ are given by
$\mu(\rho)=\rho^\alpha$ and $\lambda(\rho)=2(\alpha-1)\rho^\alpha$
they satisfy the Bresch--Desjardins (BD) constraint, which further yields the associated BD entropy estimate. This additional entropy structure plays a crucial role in deriving uniform a priori estimates, especially in the presence of large initial data and possible vacuum states.
\end{Remark}
The compressible Navier--Stokes--Korteweg system describes the motion of viscous compressible fluids with internal capillarity. It originates from the pioneering works of Van der Waals and Korteweg \cite{1,2}, and was rigorously derived by Dunn and Serrin within the framework of second-gradient theory \cite{3}. Compared with the classical compressible Navier--Stokes equations, this system involves additional analytical difficulties due to the highly nonlinear terms induced by the Korteweg stress tensor. In particular, the momentum equation contains third-order derivatives of the density, and, in the non-isentropic case, also second-order derivatives of the temperature. Meanwhile, the system has important applications in phase transitions, contact angle problems, and quantum fluid models \cite{4,5}.

In recent years, significant progress has been made in the mathematical analysis of the compressible Navier--Stokes--Korteweg system. For small initial data, Danchin and Desjardins \cite{6} established the global existence and uniqueness of strong solutions in critical spaces for the $n(\geq 2)$-dimensional isentropic model, and obtained global weak solutions in two dimensions. Haspot \cite{7,8} extended these results to the non-isentropic case, while Hattori and Li \cite{9,10} proved the global existence of smooth solutions in Sobolev spaces. The time-decay behavior was further studied in \cite{12,13}, and the case with external forces was considered in \cite{14,15,16}. The nonlinear stability of basic wave patterns can be found in \cite{17,18,19}. For large initial data, Bresch, Desjardins, and Lin \cite{5} obtained global weak solutions in the three-dimensional periodic case. Tsyganov \cite{20}, Charve and Haspot \cite{21}, and Germain and LeFloch \cite{22} investigated one-dimensional models under various assumptions on the coefficients. More recently, Chen, Zhao, and collaborators \cite{23,24,25,26} studied the global existence and large-time behavior of strong solutions in one dimension with density- or temperature-dependent coefficients.

To the best of our knowledge, the study of spherically symmetric solutions to the compressible Navier--Stokes--Korteweg system remains open. The main purpose of this paper is to address this problem by establishing the global existence of such solutions with large initial data in multidimensional exterior domains.

However, for the classical compressible Navier--Stokes equations, spherically or cylindrically symmetric solutions have been extensively studied. Under spherical symmetry, Nikolaev \cite{29} established the global-in-time existence of (generalized) solutions in a bounded annular domain, while Chen and Kratka \cite{30} investigated the flow between a stationary solid core and a free boundary connecting to the surrounding vacuum. In the cylindrically symmetric case, Frid and Shelukhin \cite{31} obtained global solvability for large initial data in a bounded annular domain. Subsequently, Hoff and Jenssen \cite{32} proved the global existence of weak solutions with large discontinuous initial data in both spherically and cylindrically symmetric settings. Moreover, the arguments in \cite{41,31} can be extended to the case of constant viscosity and temperature-dependent heat conductivity, see \cite{33,34,35,36}. For one-dimensional flows (or equivalently spherically or cylindrically symmetric flows) in bounded domains, it has been shown that global solutions converge exponentially to a constant state as time tends to infinity; see \cite{37,38,39,40}. It is worth emphasizing that the boundedness of the domain plays a crucial role in these results. The system satisfying \eqref{1.18} was first proposed by Vaigant and Kazhikhov \cite{41}, who established the global well-posedness of classical solutions to the two-dimensional periodic problem with non-vacuum and arbitrarily large initial data. In the multidimensional case, a remarkable framework was developed through a series of works by Bresch and Desjardins \cite{42,43}, initiated with Lin \cite{5} in the context of the Navier--Stokes--Korteweg system with linear shear viscosity. This framework provides additional information on the gradient of a function of $\rho$ when the viscosity coefficients satisfy the so-called Bresch--Desjardins entropy structure. In the exterior domain of a ball, Cao, Li, and Zhu \cite{46} proved the global existence of three-dimensional spherically symmetric regular solutions with large initial data and far-field vacuum by exploiting the BD entropy estimate.

Motivated by the above considerations, in this paper we study the multidimensional isentropic compressible Navier--Stokes--Korteweg system with large initial data of spherical symmetry . The main difficulties of this work is the derivation of uniform upper and lower bounds of the density, and the treatment of nonlinear high-order terms. To overcome these issues, we develop a series of a priori estimates induced by symmetry.

We are concerned with the global existence of spherically symmetric large solutions to the problem \eqref{1.1}--\eqref{1.3} in the unbounded exterior domain. Supplemented with spherically symmetric initial data, the solution $(\rho, \mathbf{u})$ to problem (\ref{1.1})--(\ref{1.3}) is also spherically symmetric, i.e.
\begin{eqnarray}\label{1.7}
(\rho(t, \xi), \mathbf{u}(t, \xi)) = \left( \hat{\rho}(t, r), \frac{\xi}{r} \hat{u}(t, r)\right), \quad r := |\xi| \ge a.
\end{eqnarray}
Accordingly, the system for $(\hat{\rho},\hat{u})$ can be reformulated as
\begin{equation}\label{1.8}
\left\{
\begin{aligned}
    &\hat{\rho}_t + \frac{1}{r^m}(r^m\hat{\rho} \hat u)_r = 0, \\
    &(\hat{\rho} \hat u)_t + \frac{1}{r^m}(r^m\hat{\rho} \hat u^2)_r + p(\hat{\rho})_r = \left[ (2\mu(\hat{\rho})+\lambda(\hat{\rho}))\frac{1}{r^m}(r^m \hat u)_r \right]_r - 2\mu(\hat{\rho})_r \frac{m}{r}\hat u \\
    &\hspace{13em}~~ + \left[ \hat{\rho}\kappa(\hat\rho)\frac{1}{r^m}(r^m\hat{\rho}_r)_r + \frac{\hat\rho\kappa'(\hat\rho)+\kappa(\hat\rho)}{2}\hat{\rho}_r^2 \right]_r - \frac{1}{r^m}(r^m\kappa(\hat\rho)\hat{\rho}_r^2)_r,
\end{aligned}
\right.
\end{equation}
and the initial-boundary condition \eqref{1.3} is reduced to
\begin{eqnarray}\label{1.9}
\left\{\begin{array}{ll}
          (\hat{\rho},\hat u)\big|_{t=0}=(\hat{\rho}_0,\hat u_0)(r), \,\,\,\,& r\geq a,\\[2mm]
          (\partial_r\hat{\rho},\hat u)\big|_{r=a}=0, \,\,\,\,& t\ge 0,
\end{array}\right.
\end{eqnarray}
where $m=d-1$ denotes the geometric factor related to the spatial dimension $d$. The initial and boundary data are assumed to be compatible.

To establish the global existence, it is convenient to reformulate the initial--boundary value problem \eqref{1.8}--\eqref{1.9} in Lagrangian coordinates. To this end, we introduce the Lagrangian variables $(t,x)$ and denote
\begin{equation}\label{1.10}
(\widetilde{\rho},\widetilde{u})(t,x)=(\hat{\rho},\hat u)(t,r),
\end{equation}
where
\begin{equation}\label{1.11}
    r = r(t, x) = r_{0}(x) + \int_{0}^{t} \hat u(s, r(s, x)) \mathrm{d}s,
\end{equation}
and
\begin{equation}\label{1.12}
    r_{0}(x) := h^{-1}(x), \quad h(r) := \int_{a}^{r} z^m \hat\rho_{0}(z) dz.
\end{equation}
Note that the function $h(r)$ is invertible on $[a,\infty)$ provided that $\hat{\rho}_0(z)>0$ for all $z\in[a,\infty)$. In view of \eqref{1.8}$_1$, \eqref{1.9}$_2$, and \eqref{1.11}, we obtain
\begin{equation*}
    \frac{\partial}{\partial t} \int_{a}^{r(t,x)} z^m \hat\rho(t, z) dz = 0.
\end{equation*}
Then it is easy to check
\begin{equation}\label{1.13}
    \int_{a}^{r(t,x)} z^m \hat\rho(t, z) dz = h(r_0(x)) = x,  \,\,\quad r(t,0) = a.
\end{equation}

As a consequence, the region $\{(t,r):\, t \geq 0,\ a \leq r < \infty\}$ is transformed into $\{(t,x):\, t \geq 0,\ 0 \leq x < \infty\}$. For simplicity, we henceforth denote $(\widetilde{\rho},\widetilde{u})$ by $(\rho,u)$. The identities (\ref{1.11}) and (\ref{1.13}) imply
\begin{equation}\label{1.14}
   r_t(t, x) = u(t, x), \quad r_x(t, x) = r^{-m}v(t, x),
\end{equation}
where $v=\dfrac{1}{\rho}$ is the specific volume. In view of (\ref{1.14}), system (\ref{1.8}) can be reformulated as
\begin{equation}\label{1.15}
\left\{
\begin{aligned}
    &v_t  = \left(r^mu\right)_x, \\
    &u_t + r^m\left[p(v)\right]_x = r^m\left[ \left(2\mu(v)+\lambda(v)\right)\frac{\left(r^m u\right)_x}{v}-\left(\frac{r^{2m}\kappa(v)}{v^{5}}v_x\right)_x+\frac{v\kappa'(v)-5\kappa(v)}{2v^{6}}r^{2m}v_x^2 \right]_x  \\
    &\hspace{7em}~~ -2mr^{m-1}\left(\mu(v)\right)_xu-m\frac{\kappa(v)}{v^{5}}r^{2m-1}v_x^2,
\end{aligned}
\right.
\end{equation}
where $t > 0$ and $x \in \mathbb{R}_+ := (0, \infty)$. The corresponding initial and boundary conditions are
\begin{eqnarray}\label{1.16}
\left\{\begin{array}{ll}
          (v, u)\big|_{t=0}=(v_0, u_0)(x), \,\,\,\,& x\geq 0,\\[2mm]
          (v_x, u)\big|_{x=0}=0, \,\,\,\,& t\ge 0.
\end{array}\right.
\end{eqnarray}
The initial data are supposed to be compatible with boundary conditions \eqref{1.16}$_2$ and satisfy the following far-field condition:
\begin{equation}\label{1.17}
\lim\limits_{x\to+\infty}(v_0,u_0)=(1.0),
\end{equation}
Accordingly, the conditions \eqref{1.5} and \eqref{1.6} are transformed into
\begin{eqnarray}
&&p(v)=v^{-\gamma},\quad \mu(v)=\widetilde{\mu},\quad \lambda(v)=\widetilde{\lambda}v^{-\alpha},\quad \kappa(v)=\widetilde{\kappa}v^{-\beta},\label{1.18}\\[2mm]
&&p(v)=v^{-\gamma}, \quad \mu(v)=\widetilde{\mu}v^{-\alpha},\quad \lambda(v)=\widetilde{\lambda}v^{-\alpha},\quad \kappa(v)=\widetilde{\kappa}v^{-\beta}.\label{1.19}
\end{eqnarray}
For simplicity and without loss of generality, we assume $\tilde{\kappa}=1$ in what follows.

We are now in a position to state the main result of this paper, namely, the global existence result for problem \eqref{1.15}--\eqref{1.17}. Clearly, this result is equivalent to a corresponding statement for problem \eqref{1.8}--\eqref{1.9} in the spherically symmetric Eulerian coordinates. Depending on the different values of the viscosity coefficients $\mu(\rho)$ and $\lambda(\rho)$ given in \eqref{1.18} and \eqref{1.19}, we obtain two main results in this paper. Our first result is concerned with the Kazhikhov model, which can be stated as follows.
\begin{Theorem}\label{T1.1}
For the Kazhikhov model, the viscosity coefficients $\mu(v)$, $\lambda(v)$, the thermal conductivity coefficient $\kappa(v)$, and the pressure function $p(v)$ satisfy (\ref{1.18}). We assume that the parameters $\alpha$, $\beta$ and $\gamma$ satisfy one of the following conditions:
\begin{itemize}
    \item[(i)] $\alpha\in\mathbb{R},\,\, -3\leq\beta\leq-2,\,\,\gamma\geq1;$\\
    \item[(ii)] $\dfrac{\beta+2-\sqrt{-2(\beta+2)(\beta+5)}}{2}<\alpha\leq\beta+3,\,\,\dfrac{-7-\sqrt{3}}{2}\leq\beta<-3,\,\, \gamma>-\beta-2$;\\
    \item[(iii)] $\dfrac{-3-\sqrt{-2\beta^2-22\beta-59}}{2}<\alpha\leq\beta+3,\,\,-5\leq\beta\leq-\dfrac{14}{3},\,\,\gamma>-\beta-2.$
\end{itemize}
Further, suppose that the initial data $(v_0, u_0)(x)$ satisfy
\begin{equation}\label{1.20}
\begin{cases}
    v_0(x)-1\in H^1(\mathbb{R}^+),\,u_0(x)\in H^1(\mathbb{R}^+),\,v_{0x}(x)\in H_r^1({\mathbb{R}^+}),\\[4pt]
    \underline{V}\leq v_0(x)\leq\overline{V},~~~~~~~~~~\forall x\ge 0,
   \end{cases}
\end{equation}
where $\underline{V}$ and $\overline{V}$ are arbitrary positive constants. The initial-boundary value problem  \eqref{1.15}--\eqref{1.17} has a unique global strong solution $(v, u)(t, x)$, which satisfies for any given constant $T > 0$ that
\begin{align}
&~~~~~~~~~~C_1^{-1}\le v(t,x) \le C_1,  \quad \forall (t,x) \in [0,T] \times \mathbb{R}^+,\label{1.21}\\
    &\|v(t)-1\|_1^2 +\|u(t)\|_1^2+ \|v_x(t)\|_{1,r}^2+\int_0^t \left( \|v_x(\tau)\|^2 + \|v_{xx}(\tau)\|_{1,r}^2+\|u_x(\tau)\|_{1,r}^2\right) d\tau \label{1.22}\\
    \leq& C_{2}(T)\left(\|v_0-1\|_1^2 +\|u_{0}\|_1^2+ \|v_{0x}\|_{1,r}^2+1\right). \notag
\end{align}
Here, $C_1$ is a positive constant depending only on $\alpha, \beta, \gamma, \underline{V}, \overline{V}, \|v_0-1\|, \|u_{0}\|$ and $\|v_{0x}\|_{0,r}$, and $C_2(T)$ is a positive constant depending only on $T, \alpha, \beta, \gamma, \underline{V}, \overline{V}, \|v_0-1\|_1, \|u_{0}\|_1$ and $\|v_{0x}\|_{1,r}$.
\end{Theorem}
Our second theorem concerns the density-dependent viscosity model. In this case, we establish the global existence of spherically symmetric solutions to the initial--boundary value problem \eqref{1.15}--\eqref{1.17} around the constant state $(1,0)$.
\begin{Theorem}\label{T1.2}
For the density-dependent viscosity model, the viscosity coefficients $\mu(v)$, $\lambda(v)$, the thermal conductivity coefficient $\kappa(v)$, and the pressure function $p(v)$ satisfy (\ref{1.19}). We assume that the parameters $\alpha$, $\beta$ and $\gamma$ satisfy one of the following conditions:
\begin{itemize}
    \item[(i)] $\alpha\in\mathbb{R},\,\,-3\leq\beta\leq-2,\,\, \gamma\geq1;$\\
    \item[(ii)] $\alpha=\frac{\beta+3}{2},\,\,\beta<-3,\,\,\gamma>-\beta-2;$\\
    \item[(iii)] $\dfrac{\beta+2-\sqrt{-2(\beta+2)(\beta+5)}}{2}<\alpha\leq\dfrac{\beta+4}{3}\,\,-4\leq\beta<-3,\,\, \gamma>-\beta-2;$\\
    \item[(iv)] $\dfrac{\beta+2-\sqrt{-2(\beta+2)(\beta+5)}}{2}<\alpha<\dfrac{\beta+2+\sqrt{-2(\beta+2)(\beta+5)}}{2},\,\,-5\leq\beta<-4,\,\, \gamma>-\beta-2;$
    \item[(v)] $\dfrac{-3-\sqrt{-2\beta^2-22\beta-59}}{2}<\alpha<\dfrac{-3+\sqrt{-2\beta^2-22\beta-59}}{2},\,\,\dfrac{-11-\sqrt{3}}{2}\leq\beta\leq\dfrac{-11+\sqrt{3}}{2},\,\,\gamma>-\beta-2.$\\
\end{itemize}
Further, suppose that the initial data $(v_0, u_0)(x)$ satisfy \eqref{1.20}. Then the initial-boundary value problem  \eqref{1.15}--\eqref{1.17} has a unique global strong solution $(v, u)(t, x)$, which satisfies for any given constant $T > 0$ that
\begin{align}
&~~~~~~~~~~C_3^{-1}\le v(t,x) \le C_3,  \quad \forall (t,x) \in [0,T] \times \mathbb{R}^+,\label{1.23}\\
    &\|v(t)-1\|_1^2 +\|u(t)\|_1^2+ \|v_x(t)\|_{1,r}^2+\int_0^t \left( \|v_x(\tau)\|^2 + \|v_{xx}(\tau)\|_{1,r}^2+\|u_x(\tau)\|_{1,r}^2\right) d\tau \label{1.22}\\
    \leq& C_{4}(T)\left(\|v_0-1\|_1^2 +\|u_{0}\|_1^2+ \|v_{0x}\|_{1,r}^2+1\right). \notag
\end{align}
Here, $C_3$ is a positive constant depending only on $\alpha, \beta, \gamma, \underline{V}, \overline{V}, \|v_0-1\|, \|u_{0}\|$ and $\|v_{0x}\|_{0,r}$, and $C_4(T)$ is a positive constant depending only on $T, \alpha, \beta, \gamma, \underline{V}, \overline{V}, \|v_0-1\|_1, \|u_{0}\|_1$ and $\|v_{0x}\|_{1,r}$.
\end{Theorem}
Some remarks on Theorems \ref{T1.1}-\ref{T1.2} are given in the following.
\begin{Remark}
To the best of our knowledge, the present work appears to be the first to address spherically symmetric solutions to the Navier--Stokes--Korteweg system.
\end{Remark}
\begin{Remark}
Theorem \ref{T1.2} includes the BD entropy case. In particular, taking $\tilde{\mu}=1$ and $\tilde{\lambda}=2(\alpha-1)$, and recalling the constitutive relations \eqref{1.4}, we require $\alpha>0$. Under these choices, the BD entropy model is recovered, and the corresponding assumptions on the parameters $\alpha, \beta$ and $\gamma$ given by
\begin{itemize}
    \item[(i)] $\alpha>0,\,\,-3\leq\beta\leq-2,\,\, \gamma\geq1;$
    \item[(ii)] $0<\alpha<\dfrac{\beta+4}{3},\,\,-4<\beta<-3,\,\,\gamma>-\beta-2;$
\end{itemize}
\end{Remark}
\begin{Remark}
Comparing the parameter ranges in Theorems \ref{T1.1} and \ref{T1.2}, the main difference arises from the different terms appearing in the estimation process of $\left\|\dfrac{v_x}{v^{\alpha+1}}\right\|$. In the case of Theorem \ref{T1.1}, due to the structural differences of the model, in the estimation process we need to deal with $\int_0^t\int_{\mathbb{R}^+}\left(\tilde{\mu}\frac{(r^m u)_x v_x^2}{v^{\alpha+3}} - \tilde{\mu}\frac{(r^m u)_x v_{xx}}{v^{\alpha+2}}\right)dxd\tau$ and the estimate of this term requires the additional condition $\alpha \leq \beta + 3$. Consequently, under this restriction, when $\alpha=\frac{\beta+3}{2}$, the Kazhikhov model admits no admissible parameter range. In contrast, for the density-dependent viscosity model, the corresponding terms $\int_0^t\int_{\mathbb{R}^+}\left(m\tilde{\mu}\frac{u}{v^\alpha r}\left(\frac{v_x}{v^{\alpha+1}}\right)_x - m\tilde{\mu}\frac{(r^m u)_x v_x}{r^{m+1}v^{2\alpha+1}}\right)dxd\tau$ lead to the condition $\alpha \leq \frac{\beta+4}{3}$, under which admissible parameter ranges can still be obtained. Therefore, for the density-dependent viscosity model in Theorem \ref{T1.2}, the admissible parameter range additionally includes the case $\alpha=\frac{\beta+3}{2}$.
\end{Remark}
\begin{Remark}
In \cite{23}, Chen \textit{et al.} established the global existence of classical solutions to the problem around the constant state $(\overline{v},\overline{u})$ under arbitrarily large initial perturbations. In that work, only time-dependent upper and lower bounds for the specific volume $v(t,x)$ were obtained, namely, $C(T)^{-1}\leq v(t,x)\leq C(T)$. In contrast, in the present paper we also establish the global existence of spherically symmetric solutions to the initial--boundary value problem \eqref{1.15}--\eqref{1.17} around the constant state $(1,0)$. Moreover, we derive uniform (in time) bounds for the specific volume of the form $C_0^{-1}\leq v(t,x)\leq C_0$, while incurring only a mild restriction on the admissible parameter range.
\end{Remark}
\begin{Remark}
In Theorem \ref{T1.1}-\ref{T1.2}, we assume that the norms $\left(\|v_0-1\|_1 +\|u_{0}\|_1+ \|r^mv_{0x}\|+ \|r^mv_{0xx}\|\right)$ can be bounded by an arbitrarily large.
\end{Remark}
\begin{Remark}
The inability to establish the large-time behavior in the multidimensional isentropic case mainly arises from the presence of the high-dimensional terms $\int_0^t \int_{\mathbb{R}^+} m(m-1)\frac{v_x^2}{v^{\alpha+\beta+4}r^2}dx d\tau$, $\int_0^t \int_{\mathbb{R}^+} 2m\frac{u(r^m u)_x}{r^{2m+1} v^\alpha}dx d\tau$ and $\int_0^t \int_{\mathbb{R}^+} m\frac{u^2 v_x}{r^{m+1} v^{\alpha+1}}dx d\tau$. As a consequence, the analysis in the present paper is restricted to the global existence of solutions. Nevertheless, since a complete basic energy estimate can still be established, the specific volume $v(t,x)$ admits uniform (in time) upper and lower bounds by positive constants independent of $T$. It is worth noting that, in the one-dimensional case, i.e., when $d=1$ (equivalently $m=0$), the large-time asymptotic behavior of solutions can indeed be obtained.
\end{Remark}
\begin{Remark}
If the parameters $\alpha$, $\beta$, and $\gamma$ satisfy the corresponding assumptions, then, by an argument similar to that used in the proof of Lemma \ref{L2.5}, when $\beta<-3$ and $\gamma>1-2\alpha$, the specific volume $v(t,x)$ still admits both upper and lower bounds. However, the lower bound depends on time $T$, that is,
\[
C(T)\leq v(t,x)\leq C_0.
\]
Next, we briefly present the key steps of the proof.
    \begin{equation*}
\begin{aligned}
\frac{1}{v}(t,x)-\frac{1}{v}(t,a_i(t))
&= \int_{a_i(t)}^x \left(\frac{1}{v}\right)_y dy
\leq \int_i^{i+1} \left|\frac{v_x}{v^2}\right| dx \\
&\leq \int_i^{i+1} v^{\alpha-1} \left|\frac{v_x}{v^{\alpha+1}}\right| dx \\
&\leq C \left(\int_i^{i+1} v^{2\alpha-2} dx\right)^{\frac{1}{2}}
\left(\int_i^{i+1} \frac{v_x^2}{v^{2\alpha+2}} dx\right)^{\frac{1}{2}} \\
&\leq C(T) \left(\int_i^{i+1} v^{-\gamma+1} \cdot v^{2\alpha+\gamma-3} dx\right)^{\frac{1}{2}} \\
&\leq C(T) \|v\|_{L^\infty}^{\frac{2\alpha+\gamma-3}{2}} \\
&\leq
\begin{cases}
C(T), & \text{if } 2\alpha+\gamma-3 \geq 0,\\[4pt]
C(T) \left\|\frac{1}{v}\right\|_{L^\infty}^{-\frac{2\alpha+\gamma-3}{2}}, & \text{if } 2\alpha+\gamma-3 \leq 0.
\end{cases}
\end{aligned}
\end{equation*}
Moreover, by applying Lemma \ref{L2.3}, we deduce that when $\gamma>1-2\alpha$ and $\beta<-3$, the specific volume satisfies
\[
C(T)\leq v(t,x)\leq C_0,
\]
where the lower bound $C(T)$ depends on $T, \alpha, \beta, \gamma, \underline{V}, \overline{V}, \|v_0-1\|_1, \|u_{0}\|$ and $\|v_{0x}\|_{0,r}$, while $C_0$ is a positive constant independent of time and depends only on $\alpha, \beta, \gamma, \underline{V}, \overline{V}, \|v_0-1\|, \|u_{0}\|$ and $\|v_{0x}\|_{0,r}$.
\end{Remark}
\begin{Remark}
In this paper, we focus on the problem of spherically symmetric strong solutions in the exterior domain. The corresponding problem in domains containing the origin remains open and will be addressed in future work.
\end{Remark}
We now outline the main ideas in the proof of our results. The principal difficulty in proving Theorem \ref{T1.1} is to establish uniform-in-time positive lower and upper bounds for the specific volume $v(t,x)$. To this end, we employ the method developed by Y.~Kanel' \cite{28}. First, due to the presence of the Korteweg tensor, the basic energy estimate (see Lemma \ref{L2.1}) yields a term of the form $\int_{\mathbb{R}^+}\dfrac{r^{2m}}{2v^{\beta+5}}v_x^2dx$. Since the problem is posed in an exterior domain with $r>a$, this implies the estimate $\int_{\mathbb{R}^+}\dfrac{v_x^2}{v^{\beta+5}}dx$. Combining this with the estimate of $\Phi(v)$ in \eqref{2.5}, one can derive the  uniform-in-time lower and upper bounds of $v(t,x)$ in case (i) of Theorem \ref{T1.1} via Kanel's method (see Lemma \ref{L2.3}). Next, under suitable assumptions on the parameters $\alpha$ and $\beta$, we obtain the estimate $\int_{\mathbb{R}^+}\dfrac{v_x^2}{v^{2\alpha+2}}dx$ from equation \eqref{2.1}$_2$ (see Lemma \ref{L2.4}). In the course of this estimate, a delicate term $\int_0^t\int_{\mathbb{R}^+}\dfrac{r^{2m}}{v^{\alpha+\beta+8}}v_x^4dxd\tau$ arises. This term is controlled by combining the Sobolev inequality and Gronwall's inequality with the estimates of $\int_{\mathbb{R}^+}\dfrac{v_x^2}{v^{\beta+5}}dx$ and $\int_0^t\int_{\mathbb{R}^+}\dfrac{[(r^m u)_x]^2}{v}dxd\tau$ (see \eqref{2.4}).  Combining the parameter restrictions obtained in Lemma \ref{L2.1} with the basic energy inequality and Kanel's method yields the uniform bounds for $v(t,x)$ in cases (ii) and (iii) of Theorem \ref{T1.1}. Once these uniform-in-time lower and upper bounds are established, higher-order energy estimates follow from the lower-order ones by means of Gronwall's inequality. Finally, Theorem \ref{T1.1} is obtained via a standard continuation argument based on the local existence result and the a priori estimates. The proof of Theorem \ref{T1.2} follows along similar lines.

The remainder of this paper is organized as follows. In Section 2, we establish a series of a priori estimates and derive uniform-in-time upper and lower bounds for the specific volume $v(t,x)$, which enable us to complete the proof of Theorem \ref{T1.1}. Section 3 is devoted to the proof of Theorem \ref{T1.2}, which follows by arguments analogous to those in Section 2.

{\bf Notations.} In this paper, $C$ stands for a  positive constant which may  change from line to line and is independent of the time $t$, and  $C(\cdot,\cdots,\cdot)$ or  $C_i(\cdot,\cdots,\cdot)(i\in {\mathbb{N}})$ denotes the constant which depends explicitly on the  qualities listed in brackets.  $L^p(\mathbb{R}^+)(1\leq p<\infty)$ is the space of measurable functions whose $p$-powers are integrable on $\mathbb{R}^+$ with its norm $\|\cdot\|_{L^p}=({\int_{{\mathbb{R}^+}}|\cdot|^pdx})^{1/p}$. For notational simplicity, we denote $\|\cdot\|_{L^2}$ by $\|\cdot\|$ when $p=2$. $L^\infty (\mathbb{R}^+)$ is the space of bounded measurable functions on $\mathbb{R}^+$ with the norm  $\|\cdot\|_{L^{\infty}}=ess\sup_{x\in\mathbb{R}^+}|\cdot|$. For an integer $k\geq0$, $H^k=H^k(\mathbb{R}^+)$ represents the standard $k$-th order  Sobolev space, with its norm  denoted by  $\|f\|_{\kappa}=\left(\sum_{i=0}^{\kappa}\left\|\partial_x^i f\right \|^2\right)^{\frac{1}{2}}$. For a given constant $m=n-1\ge0$, $L_r^p(\Omega)$ $(1\leq p\leq +\infty)$
stands for the algebraically weighted $L^p$ space defined by $L_r^p(\mathbb{R}^+):=
\left\{
f\in L_{r}^p(\mathbb{R}^+)\,\bigg|\,
\|f\|_{L_r^p}^p<\infty
\right\}$, equipped with the norm $
\|f\|_{L_r^p}
=
\left(
\int_{\mathbb{R}^+} r^{2m} |f(x)|^p\,dx
\right)^{\frac{1}{p}}$. And, $H_r^k$ $(k\geq 0)$ denotes the algebraically weighted
$H^k$ space defined by $H_r^k({\mathbb{R}^+}):=
\left\{
f\,\big|\,
\partial_x^j f\in L_r^2({\mathbb{R}^+}),\,
j=0,1,\cdots,k
\right\}$, with its norm $\|f\|_{k,r}
=
\left(
\sum_{i=0}^{k}
\|\partial_x^i f\|_{L_r^2}^2
\right)^{\frac{1}{2}}.$ We also denote $\|\cdot\|_{0,r}:=\|\cdot\|_{L_r^2}$ for notational simplicity.

\section{Proof of Theorem \ref{T1.1}.}
This section is devoted to proving Theorem \ref{T1.1}. First, when the viscosity coefficients $\mu(v)$, $\lambda(v)$, the thermal conductivity coefficient $\kappa(v)$, and the pressure function $p(v)$ satisfy \eqref{1.18}, system \eqref{1.15} becomes
\begin{equation}\label{2.1}
\left\{
\begin{aligned}
    &v_t  = \left(r^mu\right)_x, \\
    &u_t + r^m\left[p(v)\right]_x = r^m\left[ \left(2\tilde{\mu}+\frac{\tilde{\lambda}}{v^\alpha}\right)\frac{\left(r^m u\right)_x}{v}-\left(\frac{r^{2m}}{v^{\beta+5}}v_x\right)_x-\frac{\beta+5}{2}\frac{r^{2m}}{v^{\beta+6}}v_x^2 \right]_x-m\frac{r^{2m-1}}{v^{\beta+5}}v_x^2.
\end{aligned}
\right.
\end{equation}
Then, we introduce the function space $X(0,T;m,M)$ in which the solutions to the initial--boundary value problem \eqref{2.1}, \eqref{1.16}--\eqref{1.17} are sought, defined as follows:
\begin{eqnarray*}
X(0,T;m,M)=\left\{(v,u)(t,x)\left|
\begin{array}{c}
(v(t,x)-1)\in C(0, T; H^{1}(\mathbb{R}^+)\cap C^1(0, T; L^{2}(\mathbb{R}^+)),\\[2mm]
u(t,x)\in C(0, T; H^{1}(\mathbb{R}^+)\cap C^1(0, T; H^{1}(\mathbb{R}^+)),\\[2mm]
(v_x(t,x))\in C(0, T; H^{1}_r(\mathbb{R}^+),\\[2mm]
(u_x(t,x),v_{xx}(t,x))\in L^2(0, T; H^1_r(\mathbb{R}^+),\\[2mm]
\displaystyle m\leq v(t,x)\leq M, \,\,(t,x)\in [0,T]\times\mathbb{R}^+,
\end{array}
\right.\right\}
\end{eqnarray*}
with $M\geq m>0$ and $T>0$ are some positive constants.

Based on the assumptions of Theorem \ref{T1.1}, we aim to establish the following a priori estimates for the initial--boundary value problem \eqref{2.1}, \eqref{1.16}--\eqref{1.17}, as stated in the following proposition.

\begin{Proposition}[A priori estimates]\label{P2.1}
Assume that the conditions of Theorem \ref{T1.1} hold, and let $(v,u)(t,x)\in X(0,T;m,M)$ for some positive constants $m$, $M$ and $T$ be a solution to the initial-boundary value problem \eqref{2.1}, \eqref{1.16}--\eqref{1.17}. Then the following a priori estimates hold:
\begin{align}
&~~~~~~~~~~C_1^{-1}\le v(t,x) \le C_1,  \quad \forall (t,x) \in [0,T] \times \mathbb{R}^+,\label{2.2}\\
    &\|v(t)-1\|_1^2 +\|u(t)\|_1^2+ \|v_x(t)\|_{1,r}^2+\int_0^t \left( \|v_x(\tau)\|^2 + \|v_{xx}(\tau)\|_{1,r}^2+\|u_x(\tau)\|_{1,r}^2\right) d\tau \label{2.3}\\
    \leq& C_{2}(T)\left(\|v_0-1\|_1^2 +\|u_{0}\|_1^2+ \|v_{0x}\|_{1,r}^2+1\right). \notag
\end{align}
Here, $C_1$ is a positive constant depending only on $\alpha, \beta, \gamma, \underline{V}, \overline{V}, \|v_0-1\|, \|u_{0}\|$ and $\|v_{0x}\|_{0,r}$, and $C_2(T)$ is a positive constant depending only on $T, \alpha, \beta, \gamma, \underline{V}, \overline{V}, \|v_0-1\|_1, \|u_{0}\|_1,$ and $\|v_{0x}\|_{1,r}$.
\end{Proposition}
To simplify the presentation, we work under the assumptions of Proposition \ref{P2.1} throughout the remainder of this section. First,  we provide the following basic  energy estimate.

\begin{Lemma}\label{L2.1}Assume that the conditions of Proposition \ref{P2.1} hold. Then, for $t\in[0,T]$, we have
\begin{equation}\label{2.4}
\begin{aligned}
&\int_{\mathbb{R}^+}\left(\Phi(v)+\dfrac{u^2}{2}
+\dfrac{r^{2m}}{2v^{\beta+5}}v_x^2\right)\,dx  + \int_0^t \int_{\mathbb{R}^+}
\left(2\tilde{\mu}+\dfrac{\tilde{\lambda}}{v^\alpha}\right)
\dfrac{\left[(r^m u)_x\right]^2}{v}
\,dx\,d\tau
=\epsilon_0,
\end{aligned}
\end{equation}
where the function $\Phi(v)$ is defined by
\begin{eqnarray}\label{2.5}
\Phi(v)=\int_{1}^v\left(p(1)-p(s)\right)ds,
\end{eqnarray}
with $\epsilon_0=\int_{\mathbb{R}^+}\left(\Phi(v_0)+\dfrac{u_0^2}{2}
+\dfrac{r^{2m}}{2v_0^{\beta+5}}v_{0x}^2\right)\,dx$.
\end{Lemma}
\noindent{\bf Proof.}~~Multiplying $(\ref{2.1})_1$ by $(p(1)-p(v))$,  and  $(\ref{2.1})_2$ by $u$, and then combining the resultant equations, we have
\begin{equation}\label{2.6}
\begin{aligned}
&\left(\Phi(v)+\dfrac{u^2}{2}
+\dfrac{r^{2m}}{2v^{\beta+5}}v_x^2\right)_t+
\left(2\tilde{\mu}+\dfrac{\tilde{\lambda}}{v^\alpha}\right)
\dfrac{\left[(r^m u)_x\right]^2}{v}
= R_{1x},
\end{aligned}
\end{equation}
where
\begin{equation}\label{2.7}
\begin{aligned}
R_1
= {} & - r^m u\, p(v) + r^m u + \left(2\tilde{\mu} + \frac{\tilde{\lambda}}{v^\alpha}\right)
\frac{r^m u\, (r^m u)_x}{v} + (r^mu)_x \frac{r^{2m}}{v^{\beta+5}} v_x \\
& - r^m u \left[
\left(\frac{r^{2m}}{v^{\beta+5}} v_x\right)_x
+ \frac{\beta+5}{2}\frac{r^{2m}}{v^{\beta+6}} v_x^2
\right].
\end{aligned}
\end{equation}
Integrating $(\ref{2.6})$ w.r.t. $t$ and $x$  over $[0,T]\times\mathbb{R}^+$,  and using the boundary condition $(\ref{1.16})$, we get \eqref{2.4} immediately. This completes the proof of Lemma \ref{L2.1}.

Due to \cite{47}, we have a lower bound estimate on  the energy functional  $\Phi(v)$.
\begin{Lemma}\label{L2.2}There exists a uniform positive constant $c_0$, such that for all $v>0$,  it  holds
\begin{eqnarray}\label{2.8}
\rho\Phi(v)\geq c_0\Psi(v):=c_0\frac{(1-v)^2}{1+v}.
\end{eqnarray}
\end{Lemma}
Based on   Lemma \ref{L2.2}, we now demonstrate the lower and upper bounds of the density  $v(t, x)$ by  using Kanel's approach \cite{28}.
\begin{Lemma}\label{L2.3}  If the parameters $\alpha$, $\beta$ and $\gamma$ satisfy $\alpha\in\mathbb{R},\,-3\leq\beta\leq-2,\,\gamma\ge 1$, then there exists a positive constant
$C_{5}$ depending only on $\alpha, \beta, \gamma, \underline{V},\overline{V}, \|v_0-1\|, \|u_0\|$ and $\|v_{0x}\|_{0,r}$, such that
\begin{eqnarray}\label{2.12}
C_{5}^{-1}\leq v(t,x)\leq C_{5}, \quad \forall \,(t,x) \in [0, T] \times \mathbb{R}^+.
\end{eqnarray}
\end{Lemma}

\noindent{\bf Proof:}~~Firstly, from Lemma \ref{L2.1}, we obtain the basic energy estimate \eqref{2.4}. Then we define
\begin{eqnarray*}
\overline{\Upsilon}(v)=\int_1^{v}\frac{\sqrt{\Psi(s)}s}{\frac{\beta+5}{2}}ds,
\end{eqnarray*}
so it is easy  to verify that there exist two positive constants $A_1, A_2 > 0$ such that if $-3\leq\beta\leq-2$,
\begin{eqnarray}\label{2.13}
|\overline{\Upsilon}(v)|\geq\left\{\begin{array}{ll}
          A_1\left(v^{-\frac{\beta+3}{2}}+\ln{\frac{1}{v}}\right)-A_2,~~~v\rightarrow 0,\\[2mm]
          A_1\left(v^{-\frac{\beta+2}{2}}+\ln{v}\right)-A_2,~~~v\rightarrow +\infty.
\end{array}\right.
\end{eqnarray}
On the other hand, we deduce  from $(\ref{2.4})$ and $(\ref{2.8})$ that
\begin{eqnarray}\label{2.14}
\left|\overline{\Upsilon}(v)(t,x))\right|=&&\left|\int_x^{+\infty}\frac{\partial}{\partial x} \overline{\Upsilon}(v)(t,x)dy\right|=\left|\int_x^{+\infty}\frac{\sqrt{\Psi(v)}}{v^{\frac{\beta+5}{2}}}v_ydy\right|\nonumber\\
\leq&&\int_0^{+\infty}\sqrt{\Psi(v)}\left|\frac{v_x}{v^{\frac{\beta+5}{2}}}\right|dy\\
\leq&&\left\|\sqrt{\Psi(v)}(t)\right\|\cdot\left\|\frac{v_x}{v^{\frac{\beta+5}{2}}}(t)\right\|\nonumber\\
\leq&&C\left\|\sqrt{\Phi(v)}(t)\right\|\cdot\left\|\frac{v_x}{v^{\frac{\beta+5}{2}}}(t)\right\|\nonumber\\
\leq&&C.\notag
\end{eqnarray}
From (\ref{2.13})-(\ref{2.14}), we see that  $(\ref{2.12})$ holds  for $-3\leq\beta\leq-2$. The proof of Lemma \ref{L2.3} is thus finished.
\begin{Remark}
From the proof of Lemma \ref{L2.3}, we see that if $\beta\leq-2$,  then $v(t,x)\leq C_{5}^{-1}$ for all $(t,x)\in[0,T]\times\mathbb{R}^+$, and if $\beta\geq-3$,  then $v(t,x)\geq C_{5}$ for all $(t,x)\in[0,T]\times\mathbb{R}^+$, where $C_{5}$ is a positive constant given in Lemma \ref{L2.3}.
\end{Remark}

Next, we will prove a lemma concerning the estimate of  $\int_{\mathbb{R}^+}\dfrac{v_x^2}{v^{2\alpha+2}}dx$.
\begin{Lemma}\label{L2.4} (1) If the parameters $\alpha$ and $\beta$ satisfy $-3\leq\beta\leq-2,\,\,\alpha\in\mathbb{R},$ or $\dfrac{-3-\sqrt{-2\beta^2-22\beta-59}}{2}\leq\alpha\leq\beta+3,\,\,-5\leq\beta\leq-\dfrac{14}{3}$, then it hold for $t\in[0,T]$ that
\begin{eqnarray}\label{2.15}
\int_{\mathbb{R}^+}\left(\dfrac{v_x^2}{v^{2\alpha+2}}+\frac{v_x^2}{v^{\alpha+2}}\right)dx+\int_0^t\int_{\mathbb{R}^+}\left(\frac{ v_x^2}{v^{\alpha+\gamma+2}}+\frac{r^{2m}}{v^{\alpha+\beta+6}}v_{xx}^2\right)dxd\tau\leq C_6(T),
\end{eqnarray}
(2) If the parameters $\alpha$ and $\beta$ satisfy $\dfrac{\beta+2-\sqrt{-2(\beta+2)(\beta+5)}}{2}\leq\alpha\leq\beta+3,\,\,\dfrac{-7-\sqrt{3}}{2}\leq\beta<-3$, then for all $t\in[0,T]$, we have
\begin{eqnarray}\label{2.16}
\int_{\mathbb{R}^+}\left(\dfrac{v_x^2}{v^{2\alpha+2}}+\frac{v_x^2}{v^{\alpha+2}}\right)dx+\int_0^t\int_{\mathbb{R}^+}\left(\frac{ v_x^2}{v^{\alpha+\gamma+2}}+\left[\left(\frac{r^m}{v^{\frac{\alpha+\beta+6}{2}}}v_x\right)_x\right]^2\right)dxd\tau\leq C_7(T).
\end{eqnarray}

Here $C_6(T), C_7(T)$ are two positive constants depending only on $T, \alpha, \beta, \gamma, \underline{V}, \overline{V}, \|v_0-1\|_1, \|u_0\|$ and $\|v_{0x}\|_{0,r}$.
\end{Lemma}
\noindent{\bf Proof:}~~Multiplying $(\ref{2.1})_1$ by $r^{-m}\dfrac{v_x}{v^{\alpha+1}}$, on the one hand, we obtain
\begin{eqnarray}\label{2.17}
\left(\dfrac{\tilde{\lambda}v_x^2}{2v^{2\alpha+2}}+2\tilde{\mu}\frac{v_x^2}{v^{\alpha+2}}-\frac{uv_x}{r^mv^{\alpha+1}}\right)_t+\gamma\dfrac{v_x^2}{v^{\alpha+\gamma+2}}+\frac{r^{2m}}{v^{\alpha+\beta+6}}v_{xx}^2=R_{2x}+H_1,
\end{eqnarray}
where
\begin{eqnarray*}
\begin{split}
R_2=&-\frac{uv_t}{r^mv^{\alpha+1}}+2\tilde{\mu}\frac{(r^mu)_xv_x}{v^{\alpha+2}}+\left(\frac{r^{2m}v_x}{v^{\beta+5}}\right)_x\frac{v_x}{v^{\alpha+1}}+\frac{\alpha+2\beta+11}{3}\frac{r^{2m}}{v^{\alpha+\beta+7}}v_x^3-m\frac{r^{m-1}}{v^{\alpha+\beta+5}}v_x^2,\\
H_1=&m(m-1)\frac{v_x^2}{v^{\alpha+\beta+4}r^2}+\frac{\left[(r^mu)_x\right]^2}{r^{2m}v^{\alpha+1}}-2m\frac{u(r^mu)_x}{r^{2m+1}v^\alpha}+m\frac{u^2v_x}{r^{m+1}v^{\alpha+1}}+2\tilde{\mu}\frac{(r^mu)_xv_x^2}{v^{\alpha+3}}-2\tilde{\mu}\frac{(r^mu)_xv_{xx}}{v^{\alpha+2}}\\
&-f_1(\alpha,\beta)\frac{r^{2m}}{v^{\alpha+\beta+8}}v_x^4-f_2(m,\alpha,\beta)\frac{r^{m-1}}{v^{\alpha+\beta+6}}v_x^3,
\end{split}
\end{eqnarray*}
with
\begin{eqnarray*}
\begin{split}
f_1(\alpha,\beta)=&\frac{(\beta+5)(\alpha+\beta+7)}{6}+\frac{(\alpha+1)(\beta+5)}{2}-\frac{(\alpha+\beta+6)(\alpha+\beta+7)}{3},\\
f_2(m,\alpha,\beta)=&-m(2\alpha+3)+\frac{2m(\alpha+\beta+6)}{3}+m(\alpha+\beta+5)-\frac{m(\beta+5)}{3}.
\end{split}
\end{eqnarray*}
Integrating $(\ref{2.17})$ w.r.t. $t$ and $x$  over $[0,T]\times\mathbb{R}^+$,  and using the boundary condition $(\ref{1.16})$, we get
\begin{align}\label{2.18}
&\int_{\mathbb{R}^+}\left(\dfrac{\tilde{\lambda}v_x^2}{4v^{2\alpha+2}}+2\tilde{\mu}\frac{v_x^2}{v^{\alpha+2}}\right)dx+\int_0^t\int_{\mathbb{R}^+}\left(\frac{\gamma v_x^2}{v^{\alpha+\gamma+2}}+\frac{r^{2m}}{v^{\alpha+\beta+6}}v_{xx}^2\right)dxd\tau\notag\\
\leq&\int_{\mathbb{R}^+}\left(\dfrac{3\tilde{\lambda}v_{0x}^2}{2v_0^{2\alpha+2}}+2\tilde{\mu}\frac{v_{0x}^2}{v_0^{\alpha+2}}\right)dx+C\int_{\mathbb{R}^+}\left(u_0^2+u^2\right)dx+\int_0^t\int_{\mathbb{R}^+}H_1dxd\tau,
\end{align}
where we have used the following Cauchy inequality:
\begin{eqnarray}\label{2.19}
\int_0^t\int_{\mathbb{R}^+}\left(\frac{uv_x}{r^mv^{\alpha+1}}\right)_tdxd\tau\leq\int_{\mathbb{R}^+}\left(\dfrac{\tilde{\lambda}v_x^2}{4v^{2\alpha+2}}+\dfrac{\tilde{\lambda}v_{0x}^2}{v_0^{2\alpha+2}}\right)dx+C\int_{\mathbb{R}^+}\left(u^2+u_0^2\right)dx.
\end{eqnarray}
Next, we estimate the nonlinear terms $\int_0^t\int_{\mathbb{R}^+}H_1dxd\tau$.

Case 1: $-3\leq\beta\leq-2,\,\,\alpha\in\mathbb{R}$, from Lemma \ref{L2.3}, we have $C_{5}^{-1}\leq v(t,x)\leq C_{5}$.
\begin{align}\label{2.20}
&\int_0^t\int_{\mathbb{R}^+}H_1dxd\tau\notag\\
\leq&\frac{1}{8}\int_0^t\int_{\mathbb{R}^+}\left(\gamma\dfrac{ v_x^2}{v^{\alpha+\gamma+2}}+\frac{r^{2m}}{v^{\alpha+\beta+6}}v_{xx}^2\right)dxd\tau+C\int_0^t\int_{\mathbb{R}^+}\left(\frac{\left[(r^mu)_x\right]^2}{v}+\frac{r^{2m}v_x^2}{v^{\beta+5}}+u^2\right)dxd\tau\notag\\
&+C\int_0^t\int_{\mathbb{R}^+}\left(u^4+r^{2m}v_x^4\right)dxd\tau\\
\leq&\frac{1}{8}\int_0^t\int_{\mathbb{R}^+}\left(\gamma\dfrac{ v_x^2}{v^{\alpha+\gamma+2}}+\frac{r^{2m}}{v^{\alpha+\beta+6}}v_{xx}^2\right)dxd\tau+C\int_0^t\int_{\mathbb{R}^+}\left(\frac{\left[(r^mu)_x\right]^2}{v}+\frac{r^{2m}v_x^2}{v^{\beta+5}}+u^2\right)dxd\tau\notag\\
&+C\int_0^t\|\frac{r^mu}{r^m}(\tau)\|_{L^\infty}^2\int_{\mathbb{R}^+}u^2dxd\tau+C\int_0^t\|v_x(\tau)\|_{L^\infty}^2\int_{\mathbb{R}^+}r^{2m}v_x^2dxd\tau\notag\\
\leq&\frac{1}{4}\int_0^t\int_{\mathbb{R}^+}\left(\gamma\dfrac{ v_x^2}{v^{\alpha+\gamma+2}}+\frac{r^{2m}}{v^{\alpha+\beta+6}}v_{xx}^2\right)dxd\tau+C\int_0^t\int_{\mathbb{R}^+}\left(\frac{\left[(r^mu)_x\right]^2}{v}+\frac{r^{2m}v_x^2}{v^{\beta+5}}+u^2\right)dxd\tau\notag.
\end{align}
Hence, we have \eqref{2.15} from Lemma $\ref{L2.1}$, combining \eqref{2.18}-\eqref{2.20}, and applying Gronwall's inequality.
\\[2mm]

Case 2: we let $f_1(\alpha,\beta)>0$, then we can get $\dfrac{-3-\sqrt{-2\beta^2-22\beta-59}}{2}<\alpha<\dfrac{-3+\sqrt{-2\beta^2-22\beta-59}}{2}$, where we need $\dfrac{-11-\sqrt{3}}{2}\leq\beta\leq\dfrac{-11+\sqrt{3}}{2}$.  We know $\beta<-3$, so from Remark 2.1, we have $v(t,x)\leq C_{5}$.
\begin{align}\label{2.21}
\int_0^t\int_{\mathbb{R}^+}H_1dxd\tau\leq&\frac{1}{8}\int_0^t\int_{\mathbb{R}^+}\left(\gamma\dfrac{ v_x^2}{v^{\alpha+\gamma+2}}+\frac{r^{2m}}{v^{\alpha+\beta+6}}v_{xx}^2\right)dxd\tau-\frac{f_1(\alpha,\beta)}{2}\int_0^t\int_{\mathbb{R}^+}\frac{r^{2m}}{v^{\alpha+\beta+8}}v_x^4dxd\tau\notag\\
&+C\int_0^t\int_{\mathbb{R}^+}\left(\frac{\left[(r^mu)_x\right]^2}{v^{\alpha-\beta-2}}+\frac{\left[(r^mu)_x\right]^2}{v^{\alpha+1}}+\frac{r^{2m}v_x^2}{v^{\beta+5}}+\frac{v_x^2}{v^{2\alpha+2}}+u^2+u^4\right)dxd\tau\notag\\
\leq&\frac{1}{8}\int_0^t\int_{\mathbb{R}^+}\left(\gamma\dfrac{ v_x^2}{v^{\alpha+\gamma+2}}+\frac{r^{2m}}{v^{\alpha+\beta+6}}v_{xx}^2\right)dxd\tau-\frac{f_1(\alpha,\beta)}{2}\int_0^t\int_{\mathbb{R}^+}\frac{r^{2m}}{v^{\alpha+\beta+8}}v_x^4dxd\tau\\
&+C\int_0^t\int_{\mathbb{R}^+}\left(\frac{\left[(r^mu)_x\right]^2}{v^{\alpha+1}}+\frac{r^{2m}v_x^2}{v^{\beta+5}}+\frac{v_x^2}{v^{2\alpha+2}}+u^2\right)dxd\tau\notag\\
&+C\int_0^t\|\frac{r^mu}{r^m}(\tau)\|_{L^\infty}^2\int_{\mathbb{R}^+}u^2dxd\tau\notag\\
\leq&\frac{1}{8}\int_0^t\int_{\mathbb{R}^+}\left(\gamma\dfrac{ v_x^2}{v^{\alpha+\gamma+2}}+\frac{r^{2m}}{v^{\alpha+\beta+6}}v_{xx}^2\right)dxd\tau-\frac{f_1(\alpha,\beta)}{2}\int_0^t\int_{\mathbb{R}^+}\frac{r^{2m}}{v^{\alpha+\beta+8}}v_x^4dxd\tau\notag\\
&+C\int_0^t\int_{\mathbb{R}^+}\left(\frac{\left[(r^mu)_x\right]^2}{v^{\alpha+1}}+\frac{r^{2m}v_x^2}{v^{\beta+5}}+\frac{v_x^2}{v^{2\alpha+2}}+u^2\right)dxd\tau,\notag
\end{align}
where we need $\alpha\leq\beta+3$, then same as case 1 combining \eqref{2.18}-\eqref{2.21}, and applying Gronwall's inequality, when $\dfrac{-3-\sqrt{-2\beta^2-22\beta-59}}{2}\leq\alpha\leq\beta+3,\,\,-5\leq\beta\leq\-\dfrac{14}{3}$, we can obtain \eqref{2.15}.

On the other hand, Multiplying $(\ref{2.1})_1$ by $r^{-m}\frac{v_x}{v^{\alpha+1}}$, we obtain
\begin{eqnarray}\label{2.22}
\left(\dfrac{\tilde{\lambda}v_x^2}{2v^{2\alpha+2}}+2\tilde{\mu}\frac{v_x^2}{v^{\alpha+2}}-\frac{uv_x}{r^mv^{\alpha+1}}\right)_t+\gamma\dfrac{v_x^2}{v^{\alpha+\gamma+2}}+\left[\left(\frac{r^m}{v^{\frac{\alpha+\beta+6}{2}}}v_x\right)_x\right]^2=R_{3x}+H_2,
\end{eqnarray}
where
\begin{align*}
R_3=&-\frac{uv_t}{r^mv^{\alpha+1}}+2\tilde{\mu}\frac{(r^mu)_yv_x}{v^{\alpha+2}}+\left(\frac{r^{2m}v_x}{v^{\beta+5}}\right)_x\frac{v_x}{v^{\alpha+1}}+\frac{\beta+5}{3}\frac{r^{2m}}{v^{\alpha+\beta+7}}v_x^3,\\
H_2=&m^2\frac{v_x^2}{v^{\alpha+\beta+4}r^2}+\frac{\left[(r^mu)_x\right]^2}{r^{2m}v^{\alpha+1}}-2m\frac{u(r^mu)_x}{r^{2m+1}v^\alpha}+m\frac{u^2v_x}{r^{m+1}v^{\alpha+1}}-2\tilde{\mu}\frac{(r^mu)_x}{v\frac{\alpha-\beta-2}{2}r^m}\left(\frac{r^m}{v^{\frac{\alpha+\beta+6}{2}}}v_x\right)_x\\
&+2m\tilde{\mu}\frac{(r^mu)_xv_x}{r^{m+1}v^{\alpha+1}}-\tilde{\mu}\left(\alpha+\beta+4\right)\frac{(r^mu)_xv_x^2}{v^{\alpha+3}}-f'_1(\alpha,\beta)\frac{r^{2m}}{v^{\alpha+\beta+8}}v_x^4-f'_2(m,\alpha,\beta)\frac{r^{m-1}}{v^{\alpha+\beta+6}}v_x^3,
\end{align*}
with
\begin{eqnarray*}
\begin{split}
f'_1(\alpha,\beta)=&\frac{(\beta+5)(\alpha+\beta+7)}{6}+\frac{(\alpha+1)(\beta+5)}{2}-\left(\frac{\alpha+\beta+6}{2}\right)^2,\\
f'_2(m,\alpha,\beta)=&m(\alpha+\beta+5)-2m(\alpha+1)-\frac{m(\beta+5)}{3}.
\end{split}
\end{eqnarray*}
Integrating $(\ref{2.22})$ w.r.t. $t$ and $x$  over $[0,T]\times\mathbb{R}^+$,  and using the boundary condition $(\ref{2.16})$, we get the same as \eqref{2.18} using Cauhchy inequality
\begin{eqnarray}\label{2.23}
&&\int_{\mathbb{R}^+}\left(\dfrac{\tilde{\lambda}v_x^2}{4v^{\alpha+2}}+2\tilde{\mu}\frac{v_x^2}{v^{\alpha+2}}\right)dx+\int_0^t\int_{\mathbb{R}^+}\left(\gamma\dfrac{v_x^2}{v^{\alpha+\gamma+2}}+\left[\left(\frac{r^m}{v^{\frac{\alpha+\beta+6}{2}}}v_x\right)_x\right]^2\right)dxd\tau\notag\\
\leq&&\int_{\mathbb{R}^+}\left(\dfrac{3\tilde{\lambda}v_{0x}^2}{2v_0^{2\alpha+2}}+2\tilde{\mu}\frac{v_{0x}^2}{v_0^{\alpha+2}}\right)dx+C\int_{\mathbb{R}^+}\left(u_0^2+u^2\right)dx+\int_0^t\int_{\mathbb{R}^+}H_2dxd\tau,
\end{eqnarray}
to deal with the $\int_0^t\int_{\mathbb{R}^+}H_2dxd\tau$, we have follows.

Case 3: we let
\begin{equation*}
f'_1(\alpha,\beta)=\frac{(\beta+5)(\alpha+\beta+7)}{6}+\frac{(\alpha+1)(\beta+5)}{2}-\left(\frac{\alpha+\beta+6}{2}\right)^2>0,
\end{equation*}
that is,
\begin{eqnarray*}
\dfrac{\beta+2-\sqrt{-2(\beta+2)(\beta+5)}}{2}<\alpha<\dfrac{\beta+2+\sqrt{-2(\beta+2)(\beta+5)}}{2},\,\,-5\leq\beta<-3,
\end{eqnarray*}
then it follows  that
we know $\beta<-3$, so from Remark 2.1, we have $v(t,x)\leq C_{5}$.
\begin{align}\label{2.24}
&\int_0^t\int_{\mathbb{R}^+}H_2dxd\tau\notag\\
\leq&\frac{1}{8}\int_0^t\int_{\mathbb{R}^+}\left(\gamma\dfrac{ v_x^2}{v^{\alpha+\gamma+2}}+\left[\left(\frac{r^m}{v^{\frac{\alpha+\beta+6}{2}}}v_x\right)_x\right]^2\right)dxd\tau-\frac{f'_1(\alpha,\beta)}{2}\int_0^t\int_{\mathbb{R}^+}\frac{r^{2m}}{v^{\alpha+\beta+8}}v_x^4dxd\tau\notag\\
&+C\int_0^t\int_{\mathbb{R}^+}\left(\frac{\left[(r^mu)_x\right]^2}{v^{\alpha-\beta-2}}+\frac{\left[(r^mu)_x\right]^2}{v^{\alpha+1}}+\frac{r^{2m}v_x^2}{v^{\beta+5}}+\frac{v_x^2}{v^{2\alpha+2}}+u^2+u^4\right)dxd\tau\notag\\
\leq&\frac{1}{8}\int_0^t\int_{\mathbb{R}^+}\left(\gamma\dfrac{ v_x^2}{v^{\alpha+\gamma+2}}+\left[\left(\frac{r^m}{v^{\frac{\alpha+\beta+6}{2}}}v_x\right)_x\right]^2\right)dxd\tau-\frac{f'_1(\alpha,\beta)}{2}\int_0^t\int_{\mathbb{R}^+}\frac{r^{2m}}{v^{\alpha+\beta+8}}v_x^4dxd\tau\\
&+C\int_0^t\int_{\mathbb{R}^+}\left(\frac{\left[(r^mu)_x\right]^2}{v^{\alpha+1}}+\frac{r^{2m}v_x^2}{v^{\beta+5}}+\frac{v_x^2}{v^{2\alpha+2}}+u^2\right)dxd\tau+C\int_0^t\|\frac{r^mu}{r^m}(\tau)\|_{L^\infty}^2\int_{\mathbb{R}^+}u^2dxd\tau\notag\\
\leq&\frac{1}{8}\int_0^t\int_{\mathbb{R}^+}\left(\gamma\dfrac{ v_x^2}{v^{\alpha+\gamma+2}}+\left[\left(\frac{r^m}{v^{\frac{\alpha+\beta+6}{2}}}v_x\right)_x\right]^2\right)dxd\tau-\frac{f'_1(\alpha,\beta)}{2}\int_0^t\int_{\mathbb{R}^+}\frac{r^{2m}}{v^{\alpha+\beta+8}}v_x^4dxd\tau\notag\\
&+C\int_0^t\int_{\mathbb{R}^+}\left(\frac{\left[(r^mu)_x\right]^2}{v^{\alpha+1}}+\frac{r^{2m}v_x^2}{v^{\beta+5}}+\frac{v_x^2}{v^{2\alpha+2}}+u^2\right)dxd\tau,\notag
\end{align}
where we need $\alpha\leq\beta+3$. Hence, we have \eqref{2.16} from Lemma $\ref{L2.1}$, combining \eqref{2.23}-\eqref{2.24}, and applying Gronwall's inequality.

The proof of Lemma \ref{L2.4} is thus finished.

\begin{Lemma}\label{L2.5}If the parameters $\alpha$, $\beta$ and $\gamma$ satisfy $\alpha\in\mathbb{R},\,-\gamma-2<\beta<-3,\,\gamma>1$, then there exists a positive constant
$C_{8}$ depending only on $\alpha, \beta, \gamma, \underline{V}, \overline{V}, \|v_0-1\|, \|u_0\|$ and $\|v_{0x}\|_{0,r}$ such that

\begin{eqnarray}\label{2.25}
C_{8}^{-1}\leq\rho(t,x)\leq C_{8}, \quad \forall\, (t, x) \in[0, T]\times\mathbb{R}^+.
\end{eqnarray}
\end{Lemma}

\noindent{\bf Proof:}~~First, the condition $-\gamma-2<\beta<-3$ implies that $\gamma>1$. Then from the basic energy estimate  $(\ref{2.4})$, we have for each interval $[i,i+1](i\in\mathbb{N})$,
\begin{equation}\label{2.26}
\begin{split}
&\int_i^{i+1}\left(\frac{1}{\gamma-1}v^{-\gamma+1}+v-\frac{\gamma}{\gamma-1}\right)dx=\int_i^{i+1}\Phi(v)dx\leq\epsilon_0.
\end{split}
\end{equation}
Since $\Phi(v)$ is a convex function with respect to $v$, we can obtain from (\ref{2.26}) and Jensen's inequality:
 \[\varphi\left(\frac{\int_a^bf(x)dx}{b-a}\right)\leq\frac{\int_a^b\varphi(f(x))dx}{b-a}, \quad \mbox{for $f(x)\in L^1[a,b]$},\]
that
\begin{equation}\label{2.27}
\Phi\left(\int_i^{i+1}v dx\right)\leq \epsilon_0,
\end{equation}
thus, it holds that
\begin{equation}\label{2.28}
\alpha_1\leq\int_i^{i+1}v(t,x)dx\leq\alpha_2,
\end{equation}
where $\alpha_1,\alpha_2 > 0$ are two roots of the equation $\Phi(x) =\epsilon_0$. By the mean value theorem, there exists $a_i(t) \in  [i, i + 1]$ such that  holds
\begin{equation}\label{2.29}
\alpha_1\leq v(t,a_i(t))\leq\alpha_2.
\end{equation}
Then from \eqref{2.26}, we have
\begin{equation}\label{2.30}
\int^{1+1}_iv^{1-\gamma}dx\leq C.
\end{equation}
For the case  $-\gamma-2<\beta<-3$, we obtain from Remark $2.1$ that $v(t,x)\leq C_{5}$ for all $(t,x)\in[0,T]\times\mathbb{R}^+$. Therefore, it remains to show the lower bound of $v(t,x)$ for this case, which is divided  into the following two subcases:

{\it Case I:} $\gamma+\beta>0$.  For this case, we have from $(\ref{2.29})$, $(\ref{2.30})$, Lemma \ref{L2.1} and the H\"{o}lder inequality  that
\begin{align}\label{2.31}
\frac{1}{v}(t,x)-\frac{1}{v}(t,a_i(t))=&\int_{a_i(t)}^x\left(\frac{1}{v}\right)_ydy\leq\int_i^{i+1}\left|\frac{v_x}{v^2}\right|dx\notag\\
\leq&\int_i^{i+1}|v|^\frac{\beta+1}{2}\left|\frac{v_x}{v^\frac{\beta+5}{2}}\right|dx\notag\\
\leq&C\left(\int_i^{i+1}v^{\beta+1}dx\right)^\frac{1}{2}\cdot \left(\int_i^{i+1}\frac{v_x^2}{v^{\beta+5}}dx\right)^\frac{1}{2}\\
\leq&C\left(\int_i^{i+1}v^{-\gamma+1}\cdot v^{\beta+\gamma}dx\right)^\frac{1}{2}\notag\\
\leq&C\|v^\frac{\beta+\gamma}{2}\|_{L^\infty}\leq C,\notag
\end{align}
where we have used the fact that $v(t, x) \leq C_{5}$ for all $(t, x) \in [0, T] \times \mathbb{R}^+$ in the final step of $(\ref{2.31})$. Therefore, it follows from $(\ref{2.29})$ and $(\ref{2.31})$ that $v(t, x) \geq C_{9}^{-1}$ for some positive constant $C_{9}$ depending only on $\alpha, \beta, \gamma, \underline{V}, \overline{V}, \|v_0-1\|, \|u_0\|$ and $\|v_{0x}\|_{0,r}$.

{\it Case II:} $\gamma+\beta<0$. For this case, similar to (\ref{2.31}), we have
\begin{align}\label{2.32}
\frac{1}{v}(t,x)-\frac{1}{v}(t,a_i(t))=&\int_{a_i(t)}^x\left(\frac{1}{v}\right)_ydy\leq\int_i^{i+1}\left|\frac{v_x}{v^2}\right|dx\notag\\
\leq&\int_i^{i+1}|v|^\frac{\beta+1}{2}\left|\frac{v_x}{v^\frac{\beta+5}{2}}\right|dx\notag\\
\leq&C\left(\int_i^{i+1}v^{\beta+1}dx\right)^\frac{1}{2}\cdot \left(\int_i^{i+1}\frac{v_x^2}{v^{\beta+5}}dx\right)^\frac{1}{2}\\
\leq&C\left(\int_i^{i+1}v^{-\gamma+1}\cdot v^{\beta+\gamma}dx\right)^\frac{1}{2}\notag\\
\leq&C\|\frac{1}{v}\|_{L^\infty}^{-\frac{\beta+\gamma}{2}},\notag
\end{align}
which, combined with the Young inequality and $(\ref{2.30})$, indicates that  $v\geq C_{10}^{-1}$ for some positive constant $C_{10}$ depending only on $\alpha, \beta, \gamma, \underline{V}, \overline{V}, \|v_0-1\|, \|u_0\|$ and $\|v_{0x}\|_{0,r}$.

Then (\ref{2.25}) follows by letting  $C_{8}=\max\{C_{5}, C_{9}, C_{10}\}$. This ends the  proof of Lemma $\ref{L2.5}$.

As a consequence of  Lemmas \ref{L2.1}, \ref{L2.3}-\ref{L2.5}, we can get the following corollary.

\begin{Corollary}\label{C2.1}
 It holds  that
 \begin{equation}\label{2.33}
C_1^{-1}\leq \rho(t,x)\leq C_1,~~\forall(t,x)\in[0,T]\times \mathbb{R}^+,
\end{equation}
\\[-5mm]
\begin{equation}\label{2.34}
\begin{split}
&\|(v-1,u,v_x)(t)\|^2+\|v_x(t)\|^2_{0,r}+\int_0^t\|\left((r^mu)_x,v_x\right)(\tau)\|^2d\tau+\int_0^t\|v_{xx}(\tau)\|_{0,r}^2d\tau\\
\leq& C_{11}(T)\left(\|v_0-1\|^2_1+\|u_0\|^2+\|v_{0x}\|_{0,r}^2\right),\quad \forall\,t\in [0,T],
\end{split}
\end{equation}
where the constant $C_{1}$ depending only on $T, \alpha, \beta, \gamma, \underline{V}, \overline{V}, \|v_0-1\|, \|u_0\|$ and $\|v_{0x}\|_{0,r}$; $C_{11}(T)$ depending only on $T, \alpha, \beta, \gamma, \underline{V}, \overline{V}, \|v_0-1\|_1, \|u_0\|$ and $\|v_{0x}\|_{0,r}$.
\end{Corollary}

\noindent{\bf Proof:}~~First,  (\ref{2.33}) follows immediately from Lemmas  \ref{L2.3}, \ref{L2.5} with $C_1=\max\{C_{5},C_{8}\}$.
For (\ref{2.34}), we get immediately from $(\ref{2.33})$ and Lemma \ref{L2.1}, \ref{L2.4} that
\begin{equation}\label{2.35}
\begin{split}
&\|(v-1,u,v_x)(t)\|^2+\|v_x(t)\|^2_{0,r}+\int_0^t\|\left((r^mu)_x,v_x\right)(\tau)\|^2d\tau\\
\leq& C_{12}(T)\left(\|v_0-1\|_1^2+\|u_0\|^2++\|v_{0x}(t)\|^2_{0,r}\right),\quad \forall\,t\in [0,T].
\end{split}
\end{equation}
Moreover, we derive that
\begin{equation}\label{2.36}
\begin{split}
r^{2m}v_{xx}^2=&v^{\alpha+\beta+6}\left[\left(\frac{r^m}{v^{\frac{\alpha+\beta+6}{2}}}v_x\right)_x\right]^2-m^2\frac{v^2v_x^2}{r^2}-\left(\frac{\alpha+\beta+6}{2}\right)^2\frac{v_x^4r^{2m}}{v^2}-2mvr^{m-1}v_xv_{xx}\\
&+\frac{m(\alpha+\beta+6)}{2}r^{m-1}v_x^3+\frac{\alpha+\beta+6}{2}\frac{r^{2m}v_{xx}v_x^2}{v}.
\end{split}
\end{equation}
Then it follows from Lemma \ref{L2.4} and \eqref{2.33} that
\begin{align*}
&\int_0^t\|v_{xx}(\tau)\|_{0,r}^2d\tau\\
\leq&\frac{1}{4}\int_0^t\|v_{xx}(\tau)\|_{0,r}^2d\tau+C\int_0^t\int_{\mathbb{R}^+}\left(\left[\left(\frac{r^m}{v^{\frac{\alpha+\beta+6}{2}}}v_x\right)_x\right]^2+v_x^2+r^{2m}v_x^2+r^{2m}v_x^4\right)dxd\tau\\
\leq&\frac{1}{4}\int_0^t\|v_{xx}(\tau)\|_{0,r}^2d\tau+C(T)\left(\|v_0-1\|^2_1+\|u_0\|^2+\|v_{0x}\|_{0,r}^2\right)+\int_0^t\|v_x(\tau)\|_{L^{\infty}}^2\int_{\mathbb{R}^+}r^{2m}v_x^2dxd\tau\\
\leq&\frac{1}{2}\int_0^t\|v_{xx}(\tau)\|_{0,r}^2d\tau+C(T)\left(\|v_0-1\|^2_1+\|u_0\|^2+\|v_{0x}\|_{0,r}^2\right),
\end{align*}
which leads to
\begin{eqnarray}\label{2.37}
\int_0^t\|v_{xx}(\tau)\|_{0,r}^2d\tau\leq C_{13}(T)\left(\|v_0-1\|^2_1+\|u_0\|^2+\|v_{0x}\|_{0,r}^2\right).
\end{eqnarray}
Combining (\ref{2.35})-(\ref{2.37}) yields  $(\ref{2.34})$ immediately. The proof of Corollary \ref{C2.1} is thus finished.

Now, we carry out the higher order estimates  of solutions $(v,u)(t, x)$. The estimate for $\| (v_{xx}, u_x)(t)\|$ is provided in  the following lemma.
\begin{Lemma}\label{L2.6}
There exist a positive constant $C_{14}$ depending only on $\alpha, \beta , \gamma, \underline{V}, \overline{V}, \|v_0-1\|_1, \|u_0\|, \|\frac{(r^mu_0)_x}{r^m}\|$ and $\|v_{0x}\|_{1,r}$ such that  for $t\in[0, T]$,
\begin{equation}\label{2.38}
\begin{split}
&\left\|\frac{(r^mu)_x}{r^m}(\tau)\right\|+\|v_{xx}(\tau)\|_{0,r}^2+\int_0^t\left\|(r^mu)_{xx}(\tau)\right\|^2d\tau\\
\leq &C_{14}(T)\left(\|v_0-1\|^2_1+\|u_0\|^2+\|\frac{(r^mu_0)_x}{r^m}\|^2+\|v_{0x}\|_{1,r}^2+1\right)+\varepsilon\int_0^t\int_{\mathbb{R}^+}\frac{r^{2m}}{v^{\beta+5}}v_{xxx}^2dxd\tau.
\end{split}
\end{equation}
\end{Lemma}

\noindent{\bf Proof:}~~Multiplying $(\ref{2.1})_2$ by $r^{-m}(r^mu)_{xx}$ and using $(\ref{2.1})_1$, we have
\begin{equation}\label{2.39}
\left(\frac{r^{2m}}{2v^{\beta+5}}v_{xx}^2+\frac{[(r^mu)_x]^2}{2r^{2m}}\right)_t+(2\tilde{\mu}+\frac{\tilde{\lambda}}{v^\alpha})\frac{[(r^mu)_{xx}]^2}{v}=R_{4x}+H_3,
\end{equation}
where
\begin{eqnarray*}
R_4=&&\frac{(r^mu)_t(r^mu)_x}{r^{2m}}+\frac{r^{2m}}{v^{\beta+5}}v_{xx}v_{tx},\\
H_3=&&-mu\frac{[(r^mu)_x]^2}{r^{2m+1}}+2m^2\frac{vu^2(r^mu)_x}{r^{2m+2}}-m\frac{u^2}{r^{m+1}}(r^mu)_{xx}-\gamma\frac{v_x}{v^{\gamma+1}}(r^mu)_{xx}\\
&&+\left(\frac{2\tilde\mu}{v^2}+\frac{\tilde{\lambda}(\alpha+1)}{v^{\alpha+2}}\right)v_x(r^mu)_x(r^mu)_{xx}+m\frac{r^{2m-2}}{v^{\beta+5}}uv_{xx}^2-\frac{\beta+5}{2}\frac{r^{2m}}{v^{\beta+5}}(r^mu)_xv_{xx}^2\\
&&+2m\frac{v(r^mu)_x}{r^{m+1}}\Bigg\{\left(2\tilde\mu+\frac{\tilde{\lambda}}{v^{\alpha}}\right)\frac{(r^mu)_{xx}}{v}-\left(\frac{2\tilde\mu}{v^2}+\frac{\tilde{\lambda}(\alpha+1)}{v^{\alpha+2}}\right)v_x(r^mu)_x-\frac{r^{2m}}{v^{\beta+5}}v_{xxx}\\
&&~~~~~~~~~~~~~~~~~~~~~~-4m\frac{r^{m-1}}{v^{\beta+5}}v_{xx}+2(\beta+5)\frac{r^{2m}}{v^{\beta+6}}v_xv_{xx}-2m(m-1)\frac{v_x}{r^2v^{\beta+3}}\\
&&~~~~~~~~~~~~~~~~~~~~~~+3m(\beta+4)\frac{r^{m-1}}{v^{\beta+5}}v_x^2-\frac{(\beta+5)(\beta+6)}{2}\frac{r^{2m}}{v^{\beta+7}}v_x^3+\gamma\frac{v_x}{v^{\gamma+1}}\Bigg\}\\
&&+(r^mu)_{xx}\Bigg\{2m\frac{r^{m-1}}{v^{\beta+4}}v_{xx}+2m(m-1)\frac{v_x}{r^2v^{\beta+3}}-3m(\beta+4)\frac{r^{m-1}}{v^{\beta+5}}v_x^2\\
&&~~~~~~~~~~~~~~~~~~-(\beta+5)\frac{r^{m-1}}{v^{\beta+4}}v_xv_{xx}+\frac{(\beta+5)(\beta+6)}{2}\frac{r^{2m}}{v^{\beta+7}}v_x^3\Bigg\}.
\end{eqnarray*}
Integrating $(\ref{2.39})$ over $[0,T]\times\mathbb{R}^+$, and notice that the boundary conditions $(\ref{1.16})$ imply that
\begin{eqnarray}\label{4.40}
(r^mu)_t(t,0)=0,\quad v_{tx}(t,0)=0,\quad \forall t\in[0,T],
\end{eqnarray}
we can obtain
\begin{eqnarray}\label{2.41}
&&\int_{\mathbb{R}^+}\left(\frac{r^{2m}}{2v^{\beta+5}}v_{xx}^2+\frac{[(r^mu)_x]^2}{2r^{2m}}\right)dx+\int_0^t\int_{\mathbb{R}^+}\left(2\tilde{\mu}+\frac{\tilde{\lambda}}{v^\alpha}\right)\frac{[(r^mu)_{xx}]^2}{v}dxd\tau\notag\\
\leq&&\int_{\mathbb{R}^+}\left(\frac{r^{2m}}{2v_0^{\beta+5}}v_{0xx}^2+\frac{[(r^mu_0)_x]^2}{2r^{2m}}\right)dx+\int_0^t\int_{\mathbb{R}^+}H_3dxd\tau.
\end{eqnarray}
With the use of  the Cauchy inequality, the Sobolev inequality, the Young inequality and Corollary \ref{C2.1}, we have
\begin{eqnarray}\label{2.42}
&&\int_0^t\int_{\mathbb{R}^+}H_3dxd\tau\notag\\
\leq&&\frac{1}{8}\int_0^t\int_{\mathbb{R}^+}\left(2\tilde{\mu}+\frac{\tilde{\lambda}}{v^\alpha}\right)\frac{[(r^mu)_{xx}]^2}{v}dxd\tau+\varepsilon\int_0^t\int_{\mathbb{R}^+}\frac{r^{2m}}{v^{\beta+5}}v_{xxx}^2dxd\tau\notag\\
&&+C\int_0^t\int_{\mathbb{R}^+}\Bigg\{\frac{[(r^mu)_x]^2u}{r^{2m}}+[(r^mu)_x]^2+u^4+v_x^2+r^{2m}v_{xx}^2+[(r^mu)_x]^2v_x^2+r^{2m}v_{xx}^2u\notag\\
&&~~~~~~~~~~~~~~~~~~+r^{2m}v_{xx}^2(r^mu)_x+r^{2m}[(r^mu)_x]^2v_x^2+r^{2m}v_x^4+r^{4m}v_{xx}^2v_x^2+r^{4m}v_x^6\Bigg\}dxd\tau\notag\\
\leq&&\frac{1}{8}\int_0^t\int_{\mathbb{R}^+}\left(2\tilde{\mu}+\frac{\tilde{\lambda}}{v^\alpha}\right)\frac{[(r^mu)_{xx}]^2}{v}dxd\tau+\varepsilon\int_0^t\int_{\mathbb{R}^+}\frac{r^{2m}}{v^{\beta+5}}v_{xxx}^2dxd\tau+C(T)\\
&&+C\int_0^t\|u(\tau)\|_{L^\infty}\int_{\mathbb{R}^+}\left(\frac{[(r^mu)_x]^2}{r^{2m}}+r^{2m}v_{xx}^2\right)dxd\tau+C\int_0^t\|u(\tau)\|_{L^\infty}^2\int_{\mathbb{R}^+}u^2dxd\tau\notag\\
&&+C\int_0^t\|(r^mu)_x(\tau)\|^2_{L^{\infty}}\int_{\mathbb{R}^+}v_x^2dxd\tau+C\int_0^t\|(r^mu)_x(\tau)\|_{L^\infty}\int_{\mathbb{R}^+}r^{2m}v_{xx}^2dxd\tau\notag\\
&&+C\int_0^t\|(r^mv_x)(\tau)\|^2_{L^\infty}\int_{\mathbb{R}^+}r^{2m}v_{xx}^2dxd\tau+C\int_0^t\|(r^mv_x)(\tau)\|^2_{L^\infty}\|v_x(\tau)\|^2_{L^{\infty}}\int_{\mathbb{R}^+}r^{2m}v_{x}^2dxd\tau\notag\\
\leq&&\frac{1}{4}\int_0^t\int_{\mathbb{R}^+}\left(2\tilde{\mu}+\frac{\tilde{\lambda}}{v^\alpha}\right)\frac{[(r^mu)_{xx}]^2}{v}dxd\tau+\varepsilon\int_0^t\int_{\mathbb{R}^+}\frac{r^{2m}}{v^{\beta+5}}v_{xxx}^2dxd\tau+C(T)\notag\\
&&+C\int_0^t\|(r^mu)_x(\tau)\|^2\int_{\mathbb{R}^+}\frac{[(r^mu)_x]^2}{r^{2m}}+C\int_{\mathbb{R}^+}\|v_{xx}(\tau)\|_{0,r}^4dx.\notag
\end{eqnarray}
Putting (\ref{2.42}) into (\ref{2.42}), then we can obtain (\ref{2.38}) by using Corollary 4.1 and Gronwall's inequality. The proof of Lemma $\ref{L2.6}$ is thus finished.

Finally, we estimate the term  $\int_0^t\|v_{xxx}(\tau)\|_{0,r}^2d\tau$ .
\begin{Lemma}\label{L2.7}
It holds for $t\in[0, T]$ that
\begin{equation}\label{2.43}
\int^t_0\|v_{xxx}(\tau )\|_{0,r}^2d\tau\leq C_{15}(T)\left(\|v_0-1\|^2_1+\|u_0\|^2+\|\frac{(r^mu_0)_x}{r^m}\|^2+\|v_{0x}\|_{1,r}^2+1\right),
\end{equation}
where the constant $C_{15}(T)$ depending only on $\alpha, \beta , \gamma, \underline{V}, \overline{V}, \|v_0-1\|_1, \|u_0\|, \|\frac{(r^mu_0)_x}{r^m}\|$ and $\|r^mv_{0x}\|_{1,r}$.
\end{Lemma}
\noindent{\bf Proof:}~~multiplying \eqref{2.1}$_2$ the equation by $r^{-m}v_{xxx}$, we have
\begin{equation}\label{2.44}
\frac{r^{2m}}{v^{\beta+5}}v_{xxx}^2+\left(2m\frac{uv}{r^{2m+1}}v_{xx}-\frac{(r^mu)_x}{r^{2m}}v_{xx}\right)_t=R_{5x}+H_4,
\end{equation}
where
\begin{eqnarray*}
R_5=&&-\left(\frac{u}{r^m}\right)v_{xx}+\frac{u}{r^m}v_{tx}\\
H_4=&&(r^mu)_{xx}\Bigg\{\frac{(r^mu)_{xx}}{r^{2m}}-4m\frac{(r^mu)_xv}{r^{3m+1}}-2m\frac{uv_x}{r^{2m+1}}+2m(3m+1)\frac{uv^2}{r^{3m+2}}\Bigg\}-m\frac{u^2}{r^{m+1}}v_{xxx}\\
&&+\gamma\frac{v_x}{v^{\gamma+1}}v_{xxx}+\left(\frac{2\tilde\mu}{v}+\frac{\tilde\lambda}{v^{\alpha+1}}\right)(r^mu)_{xx}v_{xxx}-\left(\frac{2\tilde\mu}{v^2}+\frac{\tilde\lambda(\alpha+1)}{v^{\alpha+2}}\right)v_x(r^mu)_xv_{xxx}\\
&&-v_{xxx}\Bigg\{4m\frac{r^{m-1}}{v^{\beta+4}}v_{xx}-2(\beta+5)\frac{r^{2m}}{v^{\beta+6}}v_xv_{xx}+2m(m-1)\frac{v_x}{v^{\beta+3}r^2}-3m(\beta+4)\frac{r^{m-1}}{v^{\beta+5}}v_x^2\\
&&~~~~~~~~~~~+\frac{(\beta+5)(\beta+6)}{2}\frac{r^{2m}}{v^{\beta+7}}v_x^3\Bigg\}.
\end{eqnarray*}
Integrating  $(\ref{2.43})$ w.r.t. $t$ and $x$  over $[0,T]\times\mathbb{R}^+$, and noticing that $v_{tx}(t,0)=0$ for all $t\in[0,T]$, and applying  the Cauchy inequality, we obtain
\begin{eqnarray}\label{2.45}
&&\int_0^t\int_{\mathbb{R}^+}\frac{r^{2m}}{v^{\beta+5}}v_{xxx}^2dxd\tau\notag\\
\leq&&C_{16}\left(\|(u,r^mv_{xx},\frac{(r^mu)_x}{r^m})(t)\|^2+\|(u_0,r^mv_{0xx},\frac{(r^mu)_{0x}}{r^m})\|^2\right)+\int_0^t\int_{\mathbb{R}^+}H_4dxd\tau.
\end{eqnarray}
Furthermore, by applying the Sobolev inequality, Cauchy inequality, Young inequality, Lemma \ref{L2.6} and Corollary \ref{C2.1}, we have
\begin{align}\label{2.46}
&\int_0^t\int_{\mathbb{R}^+}H_4dxd\tau\\
\leq&\frac{1}{2}\int_0^t\int_{\mathbb{R}^+}\frac{r^{2m}}{v^{\beta+5}}v_{xxx}^2dxd\tau+C_{17}\int_0^t\left\|\left(v_x,r^{2m}v_{xx},(r^mu)_x,(r^mu)_{xx}\right)(\tau)\right\|^2d\tau+C_{18}(T).\notag
\end{align}
Here $C_{16},C_{17}$ are  positive constants depending only on $\alpha, \beta , \gamma, \underline{V}, \overline{V}, \|v_0-1\|_1, \|u_0\|, \|\frac{(r^mu_0)_x}{r^m}\|$ and $\|v_{0x}\|_{1,r}$.  Then (\ref{2.43}) follows immediately by combining (\ref{2.45})-(\ref{2.46}),  and using Corollary $\ref{C2.1}$, Lemma \ref{L2.6} and the smallness of $\varepsilon$ such that $(C_{15}+C_{16})\varepsilon<\frac{1}{4}$. This completes the proof of Lemma \ref{L2.7}.

\paragraph{Proof of Proposition \ref{P2.1}.}
It follows from Corollary \ref{C2.1} and Lemmas \ref{L2.5}--\ref{L2.7} that there exists a positive constant $C_{19}(T)$ depending on $\alpha, \beta , \gamma, \underline{V}, \overline{V}, \|v_0-1\|_1, \|u_0\|, \|\frac{(r^mu_0)_x}{r^m}\|$ and $\|v_{0x}\|_{1,r}$ such that for $0 \leq t \leq T$,
\begin{align}
    &\|v(t)-1\|_1^2 + \|u(t)\|^2+ \|\frac{(r^mu)_x}{r^m}(t)\|^2+ \|v_x(t)\|_{1,r}^2 \notag\\
    &+\int_0^t \left( \|v_x(\tau)\|^2 + \|v_{xx}(\tau)\|_{1,r}^2+\|(r^mu)_x(\tau)\|^2+\|(r^mu)_{xx}(\tau)\|^2 \right) d\tau \\
    \leq& C_{19}(T)\left(\|v_0-1\|_1^2 + \|u_0\|^2+ \|\frac{(r^mu_0)_x}{r^m}\|^2+ \|r^mv_{0x}\|_{1,r}^2+1\right). \notag
\end{align}
So we get
\begin{align}\label{2.48}
    &\|v(t)-1\|_1^2 +\|u(t)\|_1^2+ \|v_x(t)\|_{1,r}^2+\int_0^t \left( \|v_x(\tau)\|^2 + \|v_{xx}(\tau)\|_{1,r}^2+\|u_x(\tau)\|_{1,r}^2\right) d\tau \\
    \leq& C_{20}(T)\left(\|v_0-1\|_1^2 +\|u_{0}\|_1^2+ \|v_{0x}\|_{1,r}^2+1\right), \notag
\end{align}
where $C_{20}(T)$ is a positive constant depending on $T, \alpha, \beta, \gamma, \underline{V}, \overline{V}, \|v_0-1\|_1, \|u_{0}\|_1$ and $\|v_{0x}\|_{1,r}$. Proposition \ref{P2.1} follows immediately from (\ref{2.48}).

\paragraph{Proof of Theorem \ref{T1.1}.}
~~~~Once the a priori estimates have been established, the existence of global strong or classical solutions can be obtained by a standard procedure. One first establishes the existence and uniqueness of a local strong or classical solution by means of a straightforward contraction argument and then uses the derived a priori estimates to continue this local solution globally in time. The proof of Theorem \ref{T1.1} is complete.\qed

\section{Proof of Theorem \ref{T1.2}.}
This section is devoted to proving Theorem \ref{T1.2}. First, when the viscosity coefficients $\mu(\rho)$, $\lambda(\rho)$, and the thermal conductivity coefficient $\kappa(\rho)$ satisfy \eqref{1.19}, system \eqref{1.15} becomes
\begin{equation}\label{3.1}
\left\{
\begin{aligned}
    &v_t  = \left(r^mu\right)_x, \\
    &u_t + r^m\left[p(v)\right]_x = r^m\left[ \frac{2\tilde{\mu}+\tilde{\lambda}}{v^{\alpha+1}}\left(r^m u\right)_x-\left(\frac{r^{2m}}{v^{\beta+5}}v_x\right)_x-\frac{\beta+5}{2}\frac{r^{2m}}{v^{\beta+6}}v_x^2 \right]_x+\frac{2m\tilde{\mu}\alpha}{v^{\alpha+1}}r^{m-1}uv_x\\
    &~~~~~~~~~~~~~~~~~~~~~~~~~-m\frac{r^{2m-1}}{v^{\beta+5}}v_x^2,
\end{aligned}
\right.
\end{equation}
Then, we define the following set of functions $X(0, T ; m, M )$ for which the solutions to the initial--boundary value problem \eqref{3.1}, (\ref{1.16})-(\ref{1.17}), are sought as follows:
\begin{eqnarray*}
X(0,T;m,M)=\left\{(v,u)(t,x)\left|
\begin{array}{c}
(v(t,x)-1)\in C(0, T; H^{1}(\mathbb{R}^+)\cap C^1(0, T; L^{2}(\mathbb{R}^+)),\\[2mm]
u(t,x)\in C(0, T; H^{1}(\mathbb{R}^+)\cap C^1(0, T; H^{1}(\mathbb{R}^+)),\\[2mm]
(v_x(t,x))\in C(0, T; H^{1}_r(\mathbb{R}^+),\\[2mm]
(u_x(t,x),v_{xx}(t,x))\in L^2(0, T; H^1_r(\mathbb{R}^+),\\[2mm]
\displaystyle m\leq v(t,x)\leq M, \,\,(t,x)\in [0,T]\times\mathbb{R}^+,
\end{array}
\right.\right\}
\end{eqnarray*}
with $M\geq m>0$ and $T>0$ are some positive constants.

Based on the assumptions of Theorem \ref{T1.2}, we first establish the following a priori estimates in the following proposition.

\begin{Proposition}[A priori estimates]\label{P3.1}
Assume that the conditions of Theorem \ref{T1.2} hold, and let $(v,u)(t,x)\in X(0,T;m,M)$ for some positive constants $m$, $M$, and $T$ be a solution to the initial-boundary value problem (\ref{1.3})--(\ref{1.4}). Then the following a priori estimates hold:
\begin{align}
&~~~~~~~~~~C_3^{-1}\le v(t,x) \le C_3,  \quad \forall (t,x) \in [0,T] \times \mathbb{R}^+,\label{1.23}\\
    &\|v(t)-1\|_1^2 +\|u(t)\|_1^2+ \|v_x(t)\|_{1,r}^2+\int_0^t \left( \|v_x(\tau)\|^2 + \|v_{xx}(\tau)\|_{1,r}^2+\|u_x(\tau)\|_{1,r}^2\right) d\tau \label{1.22}\\
    \leq& C_{4}(T)\left(\|v_0-1\|_1^2 +\|u_{0}\|_1^2+ \|v_{0x}\|_{1,r}^2+1\right), \notag
\end{align}
Here, $C_3$ is a positive constant depending only on $\alpha, \beta, \gamma, \underline{V}, \overline{V}, \|v_0-1\|, \|u_{0}\|$ and $\|v_{0x}\|_{0,r}$, and $C_4(T)$ is a positive constant depending only on $T, \alpha, \beta, \gamma, \underline{V}, \overline{V}, \|v_0-1\|_1, \|u_{0}\|_1$ and $\|v_{0x}\|_{1,r}$.
\end{Proposition}
For brevity, we assume throughout the following that the assumptions of Proposition \ref{P3.1} are satisfied. We begin with the following basic energy estimate.
\begin{Lemma}\label{L3.1}Assume that the conditions of Proposition 3.1 hold. Then, for $t\in[0,T]$, we have
\begin{equation}\label{3.4}
\begin{aligned}
&\int_{\mathbb{R}^+}\left(\Phi(v)+u^2
+\dfrac{r^{2m}}{v^{\beta+5}}v_x^2\right)\,dx  + \int_0^t \int_{\mathbb{R}^+}
\dfrac{\left[(r^m u)_x\right]^2}{v^{\alpha+1}}dxd\tau
\leq C_{21}.
\end{aligned}
\end{equation}
where the function $\Phi(v)$ is defined by
\begin{eqnarray}\label{3.5}
\Phi(v)=\int_{1}^v\left(p(1)-p(s)\right)ds,
\end{eqnarray}
with $C_{21}$ is a postive costant depending $\alpha, \beta, \gamma, \underline{V}, \overline{V}, \|v_0-1\|, \|u_{0}\|$ and $\|v_{0x}\|_{0,r}$.
\end{Lemma}
\noindent{\bf Proof.}~~Multiplying $(\ref{3.1})_1$ by $(p(1)-p(v))$,  and  $(\ref{3.1})_2$ by $u$, and then combining the resultant equations, we have
\begin{equation}\label{3.6}
\begin{aligned}
&\left(\Phi(v)+\dfrac{u^2}{2}
+\dfrac{r^{2m}}{2v^{\beta+5}}v_x^2\right)_t+
\left(\tilde{\lambda}+\frac{2}{m+1}\tilde{\mu}\right)
\dfrac{\left[(r^m u)_x\right]^2}{v^{\alpha+1}}+\frac{2m\tilde{\mu}}{v^{\alpha+1}}\left(r^mu_x-\frac{uv}{r}\right)^2
= R_{6x},
\end{aligned}
\end{equation}
where
\begin{equation}\label{3.7}
\begin{aligned}
R_6
= {} & - r^m u\, p(v) + r^m u + \left( 2\tilde{\mu}+\tilde{\lambda}\right)
\frac{r^m u\, (r^m u)_x}{v^{\alpha+1}} -2m\tilde{\mu}\frac{r^{m-1}u^2}{v^\alpha}+ (r^mu)_x \frac{r^{2m}}{v^{\beta+5}} v_x \\
& - r^m u \left[
\left(\frac{r^{2m}}{v^{\beta+5}} v_x\right)_x
+ \frac{\beta+5}{2}\frac{r^{2m}}{v^{\beta+6}} v_x^2
\right].
\end{aligned}
\end{equation}
Integrating $(\ref{3.6})$ w.r.t. $t$ and $x$  over $[0,T]\times\mathbb{R}^+$,  and using the boundary condition $(\ref{1.16})$, we get immediately
\begin{equation*}
\begin{aligned}
&\int_{\mathbb{R}^+}\left(\Phi(v)+\dfrac{u^2}{2}
+\dfrac{r^{2m}}{2v^{\beta+5}}v_x^2\right)\,dx  + \int_0^t \int_{\mathbb{R}^+}
\Bigg\{\left(\tilde{\lambda}+\frac{2}{m+1}\tilde{\mu}\right)
\dfrac{\left[(r^m u)_x\right]^2}{v^{\alpha+1}}+\frac{2m\tilde{\mu}}{v^{\alpha+1}}\left(r^mu_x-\frac{uv}{r}\right)^2\Bigg\}dxd\tau
= \varepsilon_0.
\end{aligned}
\end{equation*}
where $\varepsilon_0=\int_{\mathbb{R}^+}\left(\Phi(v_0)+\dfrac{u_0^2}{2}
+\dfrac{r^{2m}}{2v_0^{\beta+5}}v_{0x}^2\right)\,dx$.

And from \eqref{1.4}, we have
\begin{equation*}
\tilde{\mu}>0,~~~~~2\tilde{\mu}+d\tilde{\lambda}>0.
\end{equation*}
Then we get immediately \eqref{3.4}. This completes the proof of Lemma \ref{L3.1}.

Based on Lemma \ref{L3.1}, and following the same arguments as in Section 2 (details omitted for brevity), we obtain the following two lemmas.
\begin{Lemma}\label{L3.2}There exists a uniform positive constant $c_1$, such that for all $\rho>0$,  it  holds
\begin{eqnarray}\label{3.8}
\rho\Phi(v)\geq c_1\Psi(v):=c_1\frac{(1-v)^2}{1+v}.
\end{eqnarray}
\end{Lemma}
\begin{Lemma}\label{L3.3}
If the parameters $\alpha$, $\beta$ and $\gamma$ satisfy $\alpha\in\mathbb{R},\,-3\leq\beta\leq-2,\,\gamma\ge1$, then there exists a positive constant
$C_{22}$ depending only on $\alpha, \beta, \gamma, \underline{V}, \overline{V}, \|v_0-1\|, \|u_0\|$ and $\|v_{0x}\|_{0,r}$, such that
\begin{eqnarray}\label{3.9}
C_{22}^{-1}\leq v(t,x)\leq C_{22}, \quad \forall \,(t,x) \in [0, T] \times \mathbb{R}^+.
\end{eqnarray}
\end{Lemma}

\begin{Remark}
From the proof of Lemma \ref{L2.3}, we see that if $\beta\leq-2$,  then $v(t,x)\leq C_{22}^{-1}$ for all $(t,x)\in[0,T]\times\mathbb{R}^+$, and if $\beta\geq-3$,  then $v(t,x)\geq C_{22}$ for all $(t,x)\in[0,T]\times\mathbb{R}^+$, where $C_{22}$ is a positive constant given in Lemma \ref{L3.3}.
\end{Remark}

Next, we will prove a lemma concerning the estimate of  $\int_{\mathbb{R}^+}\dfrac{v_x^2}{v^{2\alpha+2}}dx$.
\begin{Lemma}\label{L3.4}If the parameters $\alpha$ and $\beta$ satisfy \\[2mm]
(1) $-3\leq\beta\leq-2,\,\,\alpha\in\mathbb{R},$ or $\dfrac{-3-\sqrt{-2\beta^2-22\beta-59}}{2}\leq\alpha\leq\dfrac{-3+\sqrt{-2\beta^2-22\beta-59}}{2},\,\,\dfrac{-11-\sqrt{3}}{2}\leq\beta\leq\dfrac{-11+\sqrt{3}}{2}$, then for all $t\in[0,T]$, we have
\begin{eqnarray}\label{3.10}
\int_{\mathbb{R}^+}\dfrac{v_x^2}{v^{2\alpha+2}}dx+\int_0^t\int_{\mathbb{R}^+}\left(\frac{ v_x^2}{v^{\alpha+\gamma+2}}+\frac{r^{2m}}{v^{\alpha+\beta+6}}v_{xx}^2\right)dxd\tau\leq C_{23}(T);
\end{eqnarray}
(2) $\dfrac{\beta+2-\sqrt{-2(\beta+2)(\beta+5)}}{2}\leq\alpha\leq\dfrac{\beta+4}{3},\,\,-4\leq\beta<-3$, or $\dfrac{\beta+2-\sqrt{-2(\beta+2)(\beta+5)}}{2}\leq\alpha\leq\dfrac{\beta+2+\sqrt{-2(\beta+2)(\beta+5)}}{2},\,\,-5\leq\beta<-4$,
then it hold for $t\in[0,T]$ that
\begin{eqnarray}\label{3.11}
\int_{\mathbb{R}^+}\dfrac{v_x^2}{v^{2\alpha+2}}dx+\int_0^t\int_{\mathbb{R}^+}\left(\frac{ v_x^2}{v^{\alpha+\gamma+2}}+\left[\left(\frac{r^m}{v^{\frac{\alpha+\beta+6}{2}}}v_x\right)_x\right]^2\right)dxd\tau\leq C_{24}(T);
\end{eqnarray}
(3) $\alpha=\frac{\beta+3}{2},\beta\leq-3$, then it hold for $t\in[0,T]$ that
\begin{eqnarray}\label{3.12}
\int_{\mathbb{R}^+}\dfrac{v_x^2}{v^{2\alpha+2}}dx+\int_0^t\int_{\mathbb{R}^+}\left(\frac{ v_x^2}{v^{\alpha+\gamma+2}}+\frac{r^{2m}}{v^{\beta-\alpha+4}}\left[\left(\frac{v_x}{v^{\alpha+1}}\right)_x\right]^2\right)dxd\tau\leq C_{25}(T).
\end{eqnarray}
Here $C_{23}(T), C_{24}(T), C_{25}(T)$ are three positive constants depending only on $T, \alpha, \beta, \gamma, \underline{V}, \overline{V},\|v_0-1\|_1, \|u_0\|$ and $\|v_{0x}\|_{0,r}$.
\end{Lemma}
\noindent{\bf Proof:}~~Multiplying $(\ref{2.1})_1$ by $r^{-m}\frac{v_x}{v^{\alpha+1}}$, on the one hand, we obtain
\begin{eqnarray}\label{3.13}
\left(\left(2\tilde{\mu}+\tilde{\lambda}\right)\frac{v_x^2}{2v^{2\alpha+2}}-\frac{uv_x}{r^mv^{\alpha+1}}\right)_t+\gamma\dfrac{v_x^2}{v^{\alpha+\gamma+2}}+\frac{r^{2m}}{v^{\alpha+\beta+6}}v_{xx}^2=R_{7x}+H_5,
\end{eqnarray}
where
\begin{eqnarray*}
\begin{split}
R_7=&-\frac{uv_t}{r^mv^{\alpha+1}}+\left(\frac{r^{2m}}{v^{\beta+5}}\right)_x\frac{v_x}{v^{\alpha+1}}+\frac{\alpha+2\beta+11}{3}\frac{r^{2m}}{v^{\alpha+\beta+7}}v_x^3-m\frac{r^{m-1}}{v^{\alpha+\beta+5}}v_x^2,\\
H_5=&m(m-1)\frac{v_x^2}{v^{\alpha+\beta+4}r^2}+\frac{\left[(r^mu)_x\right]^2}{r^{2m}v^{\alpha+1}}-2m\frac{u(r^mu)_x}{r^{2m+1}v^\alpha}+m\frac{u^2v_x}{r^{m+1}v^{\alpha+1}}+2m\alpha\tilde{\mu}\frac{uv_x^2}{v^{2\alpha+2}r}\\
&-f_3(\alpha,\beta)\frac{r^{2m}}{v^{\alpha+\beta+8}}v_x^4-f_4(m,\alpha,\beta)\frac{r^{m-1}}{v^{\alpha+\beta+6}}v_x^3,
\end{split}
\end{eqnarray*}
with
\begin{eqnarray*}
\begin{split}
f_3(\alpha,\beta)=&\frac{(\beta+5)(\alpha+\beta+7)}{6}+\frac{(\alpha+1)(\beta+5)}{2}-\frac{(\alpha+\beta+6)(\alpha+\beta+7)}{3}\\
f_4(m,\alpha,\beta)=&-m(2\alpha+3)+\frac{2m(\alpha+\beta+6)}{3}+m(\alpha+\beta+5)-\frac{m(\beta+5)}{3}.
\end{split}
\end{eqnarray*}
Integrating $(\ref{3.13})$ w.r.t. $t$ and $x$  over $[0,T]\times\mathbb{R}^+$,  and using the boundary condition $(\ref{1.16})$, we get
\begin{eqnarray}\label{3.14}
\begin{split}
&\int_{\mathbb{R}^+}\left(2\tilde{\mu}+\tilde{\lambda}\right)\frac{v_x^2}{4v^{2\alpha+2}}dx+\int_0^t\int_{\mathbb{R}^+}\left(\frac{\gamma v_x^2}{v^{\alpha+\gamma+2}}+\frac{r^{2m}}{v^{\alpha+\beta+6}}v_{xx}^2\right)dxd\tau\\
\leq&\int_{\mathbb{R}^+}3(2\tilde{\mu}+\tilde{\lambda})\dfrac{v_{0x}^2}{4v_0^{2\alpha+2}}dx+C\int_{\mathbb{R}^+}\left(u_0^2+u^2\right)dx+\int_0^t\int_{\mathbb{R}^+}H_5dxd\tau,
\end{split}
\end{eqnarray}
where we have used the following Cauchy inequality:
\begin{eqnarray}\label{3.15}
\int_0^t\int_{\mathbb{R}^+}\left(\frac{uv_x}{r^mv^{\alpha+1}}\right)_tdxd\tau\leq\int_{\mathbb{R}^+}\left((2\tilde{\mu}+\tilde{\lambda})\dfrac{v_x^2}{4v^{2\alpha+2}}+(2\tilde{\mu}+\tilde{\lambda})\dfrac{v_{0x}^2}{4v_0^{2\alpha+2}}\right)dx+C\int_{\mathbb{R}^+}\left(u^2+u_0^2\right)dx.
\end{eqnarray}
Next, we estimate the nonlinear terms $\int_0^t\int_{\mathbb{R}^+}H_5dxd\tau$.

Case 1: $-3\leq\beta\leq-2,\,\,\alpha\in\mathbb{R}$, from Lemma \ref{L3.3}, we have $C_{5}^{-1}\leq v(t,x)\leq C_{5}$.
\begin{align}\label{3.16}
\int_0^t\int_{\mathbb{R}^+}H_5dxd\tau\leq&\frac{1}{8}\int_0^t\int_{\mathbb{R}^+}\left(\gamma\dfrac{ v_x^2}{v^{\alpha+\gamma+2}}+\frac{r^{2m}}{v^{\alpha+\beta+6}}v_{xx}^2\right)dxd\tau+C\int_0^t\int_{\mathbb{R}^+}\left(\frac{\left[(r^mu)_x\right]^2}{v^{\alpha+1}}+\frac{r^{2m}v_x^2}{v^{\beta+5}}+u^2\right)dxd\tau\notag\\
&+C\int_0^t\int_{\mathbb{R}^+}\left(u^4+r^{2m}v_x^4\right)dxd\tau\\
\leq&\frac{1}{8}\int_0^t\int_{\mathbb{R}^+}\left(\gamma\dfrac{ v_x^2}{v^{\alpha+\gamma+2}}+\frac{r^{2m}}{v^{\alpha+\beta+6}}v_{xx}^2\right)dxd\tau+C\int_0^t\int_{\mathbb{R}^+}\left(\frac{\left[(r^mu)_x\right]^2}{v^{\alpha+1}}+\frac{r^{2m}v_x^2}{v^{\beta+5}}+u^2\right)dxd\tau\notag\\
&+C\int_0^t\|\frac{r^mu}{r^m}(\tau)\|_{L^\infty}^2\int_{\mathbb{R}^+}u^2dxd\tau+C\int_0^t\|v_x(\tau)\|_{L^\infty}^2\int_{\mathbb{R}^+}r^{2m}v_x^2dxd\tau\notag\\
\leq&\frac{1}{4}\int_0^t\int_{\mathbb{R}^+}\left(\gamma\dfrac{ v_x^2}{v^{\alpha+\gamma+2}}+\frac{r^{2m}}{v^{\alpha+\beta+6}}v_{xx}^2\right)dxd\tau+C\int_0^t\int_{\mathbb{R}^+}\left(\frac{\left[(r^mu)_x\right]^2}{v^{\alpha+1}}+\frac{r^{2m}v_x^2}{v^{\beta+5}}+u^2\right)dxd\tau\notag.
\end{align}
Hence, we have \eqref{3.10} from Lemma $\ref{L3.1}$, combining \eqref{3.14}-\eqref{3.16}, and applying Gronwall's inequality.
\\[2mm]

Case 2: we let $f_3(\alpha,\beta)>0$, then we can get $\dfrac{-3-\sqrt{-2\beta^2-22\beta-59}}{2}\leq\alpha\leq\dfrac{-3+\sqrt{-2\beta^2-22\beta-59}}{2}$, where we need $\dfrac{-11-\sqrt{3}}{2}\leq\beta\leq\dfrac{-11+\sqrt{3}}{2}$. which implies $\alpha<1$, and we know $\beta<-3$, so from Remark 3.1, we have $v(t,x)\leq C_{22}$.

In particular, for this model, the basic energy estimate is not available when $\alpha \neq 0$. Therefore, for the term
\[
\int_0^t \int_{\mathbb{R}^+} m\frac{u^2 v_x}{r^{m+1} v^{\alpha+1}} \, dx d\tau,
\]
we rewrite it as
\begin{equation*}
m\frac{u^2 v_x}{r^{m+1} v^{\alpha+1}}
= -\frac{m}{\alpha}\left(\frac{u^2}{r^{m+1} v^\alpha}\right)_x
+ \frac{2m}{\alpha}\frac{u(r^m u)_x}{v^\alpha r^{2m+1}}
- \frac{m(3m+1)}{\alpha}\frac{u^2}{v^{\alpha-1} r^{2m+2}}.
\end{equation*}
So we have
\begin{align*}
\int_0^t \int_{\mathbb{R}^+} m\frac{u^2 v_x}{r^{m+1} v^{\alpha+1}} \, dx d\tau\leq& C\int_0^t \int_{\mathbb{R}^+}\frac{[(r^mu)_x]^2}{v^{\alpha+1}}dxd\tau+C\int_0^t \int_{\mathbb{R}^+}\frac{u^2}{v^{\alpha-1}}dxd\tau\\
\leq&C\int_0^t \int_{\mathbb{R}^+}\frac{[(r^mu)_x]^2}{v^{\alpha+1}}dxd\tau+C\int_0^t \int_{\mathbb{R}^+}u^2dxd\tau,
\end{align*}
And when $\alpha = 0$, it is clear that
\begin{align*}
\int_0^t \int_{\mathbb{R}^+} m\frac{u^2 v_x}{r^{m+1} v^{\alpha+1}} \, dx d\tau\leq& C\int_0^t \int_{\mathbb{R}^+}\frac{v_x^2}{v^{2\alpha+2}}dxd\tau+C\int_0^t \int_{\mathbb{R}^+}u^4dxd\tau\\
\leq&C\int_0^t \int_{\mathbb{R}^+}\frac{v_x^2}{v^{2\alpha+2}}dxd\tau+C\int_0^t\|\frac{r^mu}{r^m}(\tau)\|_{L^\infty} \int_{\mathbb{R}^+}u^2dxd\tau\\
\leq&C\int_0^t \int_{\mathbb{R}^+}\frac{v_x^2}{v^{2\alpha+2}}dxd\tau+C\int_0^t\int_{\mathbb{R}^+}\frac{[(r^mu)_x]^2}{v^{\alpha+1}}dxd\tau,
\end{align*}
above, we obtain
\begin{align}\label{3.17}
\int_0^t \int_{\mathbb{R}^+} m\frac{u^2 v_x}{r^{m+1} v^{\alpha+1}} \, dx d\tau\leq C \int_0^t\int_{\mathbb{R}^+}\left(\frac{v_x^2}{v^{2\alpha+2}}+\frac{[(r^mu)_x]^2}{v^{\alpha+1}}+u^2\right)dxd\tau.
\end{align}
Then, we have
\begin{align}\label{3.18}
\int_0^t\int_{\mathbb{R}^+}H_5dxd\tau\leq&\frac{1}{8}\int_0^t\int_{\mathbb{R}^+}\left(\gamma\dfrac{ v_x^2}{v^{\alpha+\gamma+2}}+\frac{r^{2m}}{v^{\alpha+\beta+6}}v_{xx}^2\right)dxd\tau-\frac{f_3(\alpha,\beta)}{2}\int_0^t\int_{\mathbb{R}^+}\frac{r^{2m}}{v^{\alpha+\beta+8}}v_x^4dxd\tau\notag\\
&+C\int_0^t\int_{\mathbb{R}^+}\left(\frac{u^2 v_x}{ v^{\alpha+1}} +\frac{\left[(r^mu)_x\right]^2}{v^{\alpha-\beta-2}}+\frac{\left[(r^mu)_x\right]^2}{v^{\alpha+1}}+\frac{r^{2m}v_x^2}{v^{\beta+5}}+\frac{v_x^2}{v^{2\alpha+2}}+u^2\right)dxd\tau\\
\leq&\frac{1}{8}\int_0^t\int_{\mathbb{R}^+}\left(\gamma\dfrac{ v_x^2}{v^{\alpha+\gamma+2}}+\frac{r^{2m}}{v^{\alpha+\beta+6}}v_{xx}^2\right)dxd\tau-\frac{f_3(\alpha,\beta)}{2}\int_0^t\int_{\mathbb{R}^+}\frac{r^{2m}}{v^{\alpha+\beta+8}}v_x^4dxd\tau\notag\\
&+C\int_0^t\int_{\mathbb{R}^+}\left(\frac{\left[(r^mu)_x\right]^2}{v^{\alpha+1}}+\frac{r^{2m}v_x^2}{v^{\beta+5}}+\frac{v_x^2}{v^{2\alpha+2}}+u^2\right)dxd\tau,\notag
\end{align}
where we need $\alpha\leq\frac{\beta+4}{3}$, then same as case 1 combining \eqref{3.15}, \eqref{3.17}-\eqref{3.18}, and applying Gronwall's inequality, when $\dfrac{-3-\sqrt{-2\beta^2-22\beta-59}}{2}\leq\alpha\leq\dfrac{-3+\sqrt{-2\beta^2-22\beta-59}}{2},\,\,\dfrac{-11-\sqrt{3}}{2}\leq\beta\leq\dfrac{-11+\sqrt{3}}{2}$, we can obtain \eqref{3.16}.

On the other hand, Multiplying $(\ref{3.1})_1$ by $r^{-m}\dfrac{v_x}{v^{\alpha+1}}$, we obtain
\begin{eqnarray}\label{3.19}
\left((2\tilde{\mu}+\tilde{\lambda})\dfrac{v_x^2}{2v^{2\alpha+2}}-\frac{uv_x}{r^mv^{\alpha+1}}\right)_t+\gamma\dfrac{v_x^2}{v^{\alpha+\gamma+2}}+\left[\left(\frac{r^m}{v^{\frac{\alpha+\beta+6}{2}}}v_x\right)_x\right]^2=R_{8x}+H_6,
\end{eqnarray}
where
\begin{align*}
R_8=&-\frac{uv_t}{r^mv^{\alpha+1}}+\left(\frac{r^{2m}}{v^{\beta+5}}\right)_x\frac{v_x}{v^{\alpha+1}}+\frac{\beta+5}{3}\frac{r^{2m}}{v^{\alpha+\beta+7}}v_x^3,\\
H_6=&m^2\frac{v_x^2}{v^{\alpha+\beta+4}r^2}+\frac{\left[(r^mu)_x\right]^2}{r^{2m}v^{\alpha+1}}-2m\frac{u(r^mu)_x}{r^{2m+1}v^\alpha}+m\frac{u^2v_x}{r^{m+1}v^{\alpha+1}}+2m\alpha\tilde{\mu}\frac{uv_x^2}{v^{2\alpha+2}r}\\
&-f'_3(\alpha,\beta)\frac{r^{2m}}{v^{\alpha+\beta+8}}v_x^4-f'_4(m,\alpha,\beta)\frac{r^{m-1}}{v^{\alpha+\beta+6}}v_x^3,
\end{align*}
with
\begin{eqnarray*}
\begin{split}
f'_3(\alpha,\beta)=&\frac{(\beta+5)(\alpha+\beta+7)}{6}+\frac{(\alpha+1)(\beta+5)}{2}-\left(\frac{\alpha+\beta+6}{2}\right)^2,\\
f'_4(m,\alpha,\beta)=&m(\alpha+\beta+5)-2m(\alpha+1)-\frac{m(\beta+5)}{3}.
\end{split}
\end{eqnarray*}
Integrating $(\ref{3.19})$ w.r.t. $t$ and $x$  over $[0,T]\times\mathbb{R}^+$,  and using the boundary condition $(\ref{1.16})$, we get the same as \eqref{2.19} using Cauhchy inequality
\begin{eqnarray}\label{3.20}
\begin{split}
&\int_{\mathbb{R}^+}\left(2\tilde{\mu}+\tilde{\lambda}\right)\frac{v_x^2}{4v^{2\alpha+2}}dx+\int_0^t\int_{\mathbb{R}^+}\left(\gamma\dfrac{v_x^2}{v^{\alpha+\gamma+2}}+\left[\left(\frac{r^m}{v^{\frac{\alpha+\beta+6}{2}}}v_x\right)_x\right]^2\right)dxd\tau\\
\leq&\int_{\mathbb{R}^+}3(2\tilde{\mu}+\tilde{\lambda})\dfrac{v_{0x}^2}{4v_0^{2\alpha+2}}dx+C\int_{\mathbb{R}^+}\left(u_0^2+u^2\right)dx+\int_0^t\int_{\mathbb{R}^+}H_6dxd\tau.
\end{split}
\end{eqnarray}
To deal with the $\int_0^t\int_{\mathbb{R}^+}H_6dxd\tau$, we have follows.

Case 3: we let
\begin{equation*}
f'_3(\alpha,\beta)=\frac{(\beta+5)(\alpha+\beta+7)}{6}+\frac{(\alpha+1)(\beta+5)}{2}-\left(\frac{\alpha+\beta+6}{2}\right)^2>0,
\end{equation*}
that is,
\begin{eqnarray*}
\dfrac{\beta+2-\sqrt{-2(\beta+2)(\beta+5)}}{2}<\alpha<\dfrac{\beta+2+\sqrt{-2(\beta+2)(\beta+5)}}{2},\,\,-5\leq\beta<-3
\end{eqnarray*}
then it follows  that
we know $\beta<-3$, so from Remark 3.1, we have $v(t,x)\leq C_{22}$.
\begin{align}\label{3.21}
\int_0^t\int_{\mathbb{R}^+}H_6dxd\tau\leq&\frac{1}{8}\int_0^t\int_{\mathbb{R}^+}\left(\gamma\dfrac{ v_x^2}{v^{\alpha+\gamma+2}}+\left[\left(\frac{r^m}{v^{\frac{\alpha+\beta+6}{2}}}v_x\right)_x\right]^2\right)dxd\tau-\frac{f'_3(\alpha,\beta)}{2}\int_0^t\int_{\mathbb{R}^+}\frac{r^{2m}}{v^{\alpha+\beta+8}}v_x^4dxd\tau\notag\\
&+C\int_0^t\int_{\mathbb{R}^+}\left(\frac{u^2 v_x}{ v^{\alpha+1}} +\frac{\left[(r^mu)_x\right]^2}{v^{\alpha-\beta-2}}+\frac{\left[(r^mu)_x\right]^2}{v^{\alpha+1}}+\frac{r^{2m}v_x^2}{v^{\beta+5}}+\frac{v_x^2}{v^{2\alpha+2}}+u^2\right)dxd\tau\\
\leq&\frac{1}{8}\int_0^t\int_{\mathbb{R}^+}\left(\gamma\dfrac{ v_x^2}{v^{\alpha+\gamma+2}}+\left[\left(\frac{r^m}{v^{\frac{\alpha+\beta+6}{2}}}v_x\right)_x\right]^2\right)dxd\tau-\frac{f'_3(\alpha,\beta)}{2}\int_0^t\int_{\mathbb{R}^+}\frac{r^{2m}}{v^{\alpha+\beta+8}}v_x^4dxd\tau\notag\\
&+C\int_0^t\int_{\mathbb{R}^+}\left(\frac{\left[(r^mu)_x\right]^2}{v^{\alpha+1}}+\frac{r^{2m}v_x^2}{v^{\beta+5}}+\frac{v_x^2}{v^{2\alpha+2}}+u^2\right)dxd\tau,\notag
\end{align}
where we need $\alpha\leq\dfrac{\beta+4}{3}$. Hence, we have \eqref{3.11} from Lemma $\ref{L3.1}$, combining \eqref{3.20}-\eqref{3.21}, and applying Gronwall's inequality.

The another case, we have the following:
\begin{align}\label{3.22}
    \left((2\tilde{\mu}+\tilde{\lambda})\frac{v_x^2}{2v^{2\alpha+2}}-\frac{uv_x}{r^mv^{\alpha+1}}\right)_t+\gamma\frac{v_x^2}{v^{\alpha+\gamma+2}}+\frac{r^{2m}}{v^{\beta-\alpha+4}}\left[\left(\frac{v_x}{v^{\alpha+1}}\right)_x\right]^2=R_{9x}+H_7,
\end{align}
where
\begin{align*}
R_9=&-\frac{uv_t}{r^mv^{\alpha+1}}+\left(\frac{r^{2m}}{v^{\beta+5}}v_x\right)_x\frac{v_x}{v^{\alpha+1}}+\frac{\beta+5}{2}\frac{r^{2m}}{v^{\alpha+\beta+7}}v_x^3+2m\tilde{\mu}\frac{uv_x}{v^{2\alpha+1}r}-\frac{m}{\alpha}\frac{u^2}{r^{m+1}v^\alpha}+\frac{m}{\beta-\alpha+3}\frac{r^{m-1}v_x^2}{v^{\alpha+\beta+5}},\\
H_7=&\frac{\left[(r^mu)_x\right]^2}{r^{2m}v^{\alpha+1}}-2m\frac{u(r^mu)_x}{r^{2m+1}v^\alpha}+\frac{2m}{\alpha}\frac{u(r^mu)_x}{r^{2m+1}v^\alpha}-\frac{m(3m+1)}{\alpha}\frac{u^2}{r^{2m+2}v^{\alpha-1}}-2m\frac{r^{m-1}v_x}{v^{\beta+4}}\left(\frac{v_x}{v^{\alpha+1}}\right)_x\\
&-\left(\alpha-\frac{\beta+3}{2}\right)\frac{r^{2m}}{v^{\beta+6}}v_x^2\left(\frac{v_x}{v^{\alpha+1}}\right)_x-2m\tilde{\mu}\frac{u}{v^\alpha r}\left(\frac{v_x}{v^{\alpha+1}}\right)_x-2m\tilde{\mu}\frac{(r^mu)_xv_x}{r^{m+1}v^{2\alpha+1}}+2m(m+1)\tilde{\mu}\frac{uv_x}{r^{m+2}v^{2\alpha}}\\
&+\frac{2m}{\beta-\alpha+3}\frac{r^{m-1}v_x}{v^{\beta+4}}\left(\frac{v_x}{v^{\alpha+1}}\right)_x+\frac{m(m-1)}{\beta-\alpha+3}\frac{v_x^2}{v^{\alpha+\beta+4}r^2}.
\end{align*}
Case 4: we need $\alpha=\dfrac{\beta+3}{2},\beta<-3$, then we obtain
\begin{align}\label{3.21}
\int_0^t\int_{\mathbb{R}^+}H_5dxd\tau\leq&\frac{1}{8}\int_0^t\int_{\mathbb{R}^+}\left(\gamma\dfrac{ v_x^2}{v^{\alpha+\gamma+2}}+\frac{r^{2m}}{v^{\beta-\alpha+4}}\left[\left(\frac{v_x}{v^{\alpha+1}}\right)_x\right]^2\right)dxd\tau\notag\\
&+C\int_0^t\int_{\mathbb{R}^+}\left(\frac{\left[(r^mu)_x\right]^2}{v^{\alpha+1}}+\frac{r^{2m}v_x^2}{v^{\beta+5}}+\frac{v_x^2}{v^{2\alpha+2}}+u^2\right)dxd\tau,\notag
\end{align}
the same we need $\alpha\leq\dfrac{\beta+4}{3}$, so we finish the case 4.

The proof of Lemma \ref{L3.4} is thus finished.

\begin{Lemma}\label{L3.5}If the parameters $\alpha$, $\beta$ and $\gamma$ satisfy  $\alpha\in\mathbb{R},\,-\gamma-2<\beta<-3,\,\gamma>1$, then there exists a positive constant
$C_{26}$ depending only on $\alpha, \beta, \gamma, \underline{V}, \overline{V}, \|v_0-1\|, \|u_0\|$ and $\|v_{0x}\|_{0,r}$ such that
\begin{eqnarray}\label{3.23}
C_{26}^{-1}\leq v(t,x)\leq C_{26}, \quad \forall\, (t, x) \in[0, T]\times\mathbb{R}^+.
\end{eqnarray}
\end{Lemma}
The proof is omitted here, as it follows along the same lines as that of Lemma \ref{L2.5}.

As a consequence of  Lemmas \ref{L3.1}, \ref{L3.3}-\ref{L3.5}, we can get the following corollary.

\begin{Corollary}\label{C3.1}
 It holds  that
 \begin{equation}\label{3.24}
C_3^{-1}\leq v(t,x)\leq C_3,~~\forall(t,x)\in[0,T]\times \mathbb{R}^+,
\end{equation}
\\[-5mm]
\begin{equation}\label{3.25}
\begin{split}
&\|(v,u,v_x)(t)\|^2+\|v_x(t)\|_{0,r}^2+\int_0^t\|\left((r^mu)_x,v_x\right)(\tau)\|^2d\tau+\int_0^t\|v_{xx}(\tau)\|_{0,r}^2d\tau\\
\leq& C_{27}(T)\left(\|v_0\|^2_1+\|u_0\|^2+\|v_{0x}(t)\|_{0,r}^2\right),\quad \forall\,t\in [0,T],
\end{split}
\end{equation}
where the constant $C_{3}(T)$ depending only on $\alpha, \beta, \gamma, \underline{V}, \overline{V}, \|v_0-1\|, \|r^mv_{0x}\|$ and $\|u_0\|$; $C_{27}(T)$ depending only on $T, \alpha, \beta, \gamma, \underline{V}, \overline{V}, \|v_0-1\|_1, \|r^mv_{0x}\|$ and $\|u_0\|$.
\end{Corollary}

\noindent{\bf Proof:}~~First,  (\ref{3.24}) follows immediately from Lemmas  \ref{L3.3}, \ref{L3.5} with $C_3=\max\{C_{22},C_{26}\}$.
For (\ref{3.25}), we get immediately from $(\ref{3.24})$ and Lemma \ref{L3.1}, \ref{L3.4} that
\begin{equation}\label{3.26}
\begin{split}
&\|(v,u,v_x)(t)\|^2+\|v_x(t)\|_{0,r}^2+\int_0^t\|\left((r^mu)_x,v_x\right)(\tau)\|^2d\tau\\
\leq& C_{28}(T)\left(\|v_0\|_1^2+\|u_0\|^2+\|v_{0x}\|_{0,r}^2\right),\quad \forall\,t\in [0,T].
\end{split}
\end{equation}
Moreover, we derive that
\begin{equation}\label{3.27}
\begin{split}
r^{2m}v_{xx}^2=&v^{\alpha+\beta+6}\left[\left(\frac{r^m}{v^{\frac{\alpha+\beta+6}{2}}}v_x\right)_x\right]^2-m^2\frac{v^2v_x^2}{r^2}-\left(\frac{\alpha+\beta+6}{2}\right)^2\frac{v_x^4r^{2m}}{v^2}-2mvr^{m-1}v_xv_{xx}\\
&+\frac{m(\alpha+\beta+6)}{2}r^{m-1}v_x^3+\frac{\alpha+\beta+6}{2}\frac{r^{2m}v_{xx}v_x^2}{v},
\end{split}
\end{equation}
and
\begin{equation}\label{3.28}
\begin{split}
r^{2m}v_{xx}^2=&v^{\alpha+\beta+6}\Bigg\{\frac{r^{2m}}{v^{\beta-\alpha+4}}\left[\left(\frac{v_x}{v^{\alpha+1}}\right)\right]^2+2(\alpha+1)\frac{r^{2m}}{v^{\alpha+\beta+7}}v_x^2v_{xx}-(\alpha+1)^2\frac{r^{2m}}{v^{\alpha+\beta+8}}v_x^4\Bigg\}.
\end{split}
\end{equation}
Then it follows from Lemma \ref{L3.4} that
\begin{align*}
&\int_0^t\|v_{xx}(\tau)\|_{0,r}^2d\tau\\
\leq&\frac{1}{4}\int_0^t\|v_{xx}(\tau)\|_{0,r}^2d\tau+C\int_0^t\int_{\mathbb{R}^+}\left(\left[\left(\frac{r^m}{v^{\frac{\alpha+\beta+6}{2}}}v_x\right)_x\right]^2+v_x^2+r^{2m}v_x^2+r^{2m}v_x^4\right)dxd\tau\\
\leq&\frac{1}{4}\int_0^t\|v_{xx}(\tau)\|_{0,r}^2d\tau+C(T)\left(\|v_0\|^2_1+\|u_0\|^2+\|v_{0x}\|_{0,r}^2\right)+\int_0^t\|v_x(\tau)\|_{L^{\infty}}^2\int_{\mathbb{R}^+}r^{2m}v_x^2dxd\tau\\
\leq&\frac{1}{2}\int_0^t\|v_{xx}(\tau)\|_{0,r}^2d\tau+C(T)\left(\|v_0\|^2_1+\|u_0\|^2+\|v_{0x}\|_{0,r}^2\right),
\end{align*}
which leads to
\begin{eqnarray}\label{3.29}
\int_0^t\|v_{xx}(\tau)\|_{0,r}^2d\tau\leq C_{29}(T)\left(\|v_0\|^2_1+\|u_0\|^2+\|v_{0x}\|_{0,r}^2\right).
\end{eqnarray}
Combining (\ref{3.26})-(\ref{3.29}) yields  $(\ref{3.25})$ immediately, The proof of Corollary \ref{C3.1} is thus finished.

Now, we carry out the higher order estimates  of solutions $(v,u)(t, x)$. The estimate for $\| (v_{xx}, u_x)(t)\|$ is provided in  the following lemma.
\begin{Lemma}\label{L3.6}
There exist a positive constant $C_{30}$ depending only on $T, \alpha, \beta , \gamma, \underline{V}, \overline{V}, \|v_0-1\|_1, \|u_0\|, \|\frac{(r^mu)_x}{r^m}\|$ and $\|v_{0x}\|_{1,r}$, such that  for $t\in[0, T]$,
\begin{align}\label{3.30}
&\left\|\frac{(r^mu)_x}{r^m}(\tau)\right\|^2+\|v_{xx}(\tau)\|_{0,r}+\int_0^t\left\|(r^mu)_{xx}(\tau)\right\|d\tau\notag\\
\leq &C_{30}(T)\left(\|v_0\|_1+\|u_0\|+\|\frac{(r^mu)_x}{r^m}\|+\|v_{0x}\|_{1,r}+1\right)+\varepsilon\int_0^t\int_{\mathbb{R}^+}\frac{r^{2m}}{v^{\beta+5}}v_{xxx}^2dxd\tau.
\end{align}
\end{Lemma}

\noindent{\bf Proof:}~~Multiplying $(\ref{3.1})_2$ by $r^{-m}(r^mu)_{xx}$ and using $(\ref{3.1})_1$, we have
\begin{equation}\label{3.31}
\left(\frac{r^{2m}}{2v^{\beta+5}}v_{xx}^2+\frac{[(r^mu)_x]^2}{2r^{2m}}\right)_t+(2\tilde{\mu}+\tilde{\lambda})\frac{[(r^mu)_{xx}]^2}{v^{\alpha+1}}=R_{10x}+H_8,
\end{equation}
where
\begin{eqnarray*}
R_{10}=&&\frac{(r^mu)_t(r^mu)_x}{r^{2m}}+\frac{r^{2m}}{v^{\beta+5}}v_{xx}v_{tx},\\
H_8=&&-mu\frac{[(r^mu)_x]^2}{r^{2m+1}}+2m^2\frac{vu^2(r^mu)_x}{r^{2m+2}}-m\frac{u^2}{r^{m+1}}(r^mu)_{xx}-\gamma\frac{v_x}{v^{\gamma+1}}(r^mu)_{xx}\\
&&-\frac{2m\alpha\tilde{\mu}}{rv^{\alpha+1}}uv_x(r^mu)_x+(2\tilde{\mu}+\tilde{\lambda})(\alpha+1)\frac{(r^mu)_x(r^mu)_{xx}v_x}{v^{\alpha+2}}+m\frac{r^{2m-1}}{v^{\beta+5}}uv_{xx}^2-\frac{\beta+5}{2}\frac{r^{2m}}{v^{\beta+5}}(r^mu)_xv_{xx}^2\\
&&+2m\frac{v(r^mu)_x}{r^{m+1}}\Bigg\{\frac{2\tilde\mu+\tilde{\lambda}}{v^{\alpha+1}}(r^mu)_{xx}-\frac{(2\tilde\mu+\tilde{\lambda})(\alpha+1)}{v^{\alpha+2}}v_x(r^mu)_x-\frac{r^{2m}}{v^{\beta+5}}v_{xxx}-4m\frac{r^{m-1}}{v^{\beta+5}}v_{xx}\\
&&~~~~~~~~~~~~~~~~~~~~~~+2(\beta+5)\frac{r^{2m}}{v^{\beta+6}}v_xv_{xx}-2m(m-1)\frac{v_x}{r^2v^{\beta+3}}+3m(\beta+4)\frac{r^{m-1}}{v^{\beta+5}}v_x^2\\
&&~~~~~~~~~~~~~~~~~~~~~~-\frac{(\beta+5)(\beta+6)}{2}\frac{r^{2m}}{v^{\beta+7}}v_x^3+\gamma\frac{v_x}{v^{\gamma+1}}+\frac{2m\alpha\tilde{\mu}}{v^{\alpha+1}r}uv_x\Bigg\}\\
&&+(r^mu)_{xx}\Bigg\{2m\frac{r^{m-1}}{v^{\beta+4}}v_{xx}+2m(m-1)\frac{v_x}{r^2v^{\beta+3}}-3m(\beta+4)\frac{r^{m-1}}{v^{\beta+5}}v_x^2\\
&&~~~~~~~~~~~~~~~~~~-(\beta+5)\frac{r^{m-1}}{v^{\beta+4}}v_xv_{xx}+\frac{(\beta+5)(\beta+6)}{2}\frac{r^{2m}}{v^{\beta+7}}v_x^3\Bigg\}.
\end{eqnarray*}
The proof follows the same argument as that of Lemma \ref{L2.6} and is omitted for brevity.

Finally, we estimate the term  $\int_0^t\|v_{xxx}(\tau)\|_{0,r}^2d\tau$ .
\begin{Lemma}\label{L3.7}
It holds for $t\in[0, T]$ that
\begin{equation}\label{3.32}
\int^t_0\|v_{xxx}(\tau )\|_{0,r}^2d\tau\leq C_{31}(T)\left(\|v_0-1\|^2_1+\|u_0\|^2+\|\frac{(r^mu_0)_x}{r^m}\|^2+\|r^mv_{0x}\|_{1,r}^2+1\right),
\end{equation}
where the constant $C_{31}(T)$ depending only on $T, \alpha, \beta , \gamma, \underline{V}, \overline{V}, \|v_0-1\|_1, \|u_0\|, \|\frac{(r^mu_0)_x}{r^m}\|$ and $\|v_{0x}\|_{1,r}$.
\end{Lemma}
\noindent{\bf Proof:}~~multiplying \eqref{3.1} the equation by $r^{-m}v_{xxx}$, we have
\begin{equation}\label{3.33}
\frac{r^{2m}}{v^{\beta+5}}v_{xxx}^2+\left(2m\frac{uv}{r^{2m+1}}v_{xx}-\frac{(r^mu)_x}{r^{2m}}v_{xx}\right)_t=R_{11x}+H_9,
\end{equation}
where
\begin{eqnarray*}
R_{11}=&&-\left(\frac{u}{r^m}\right)v_{xx}+\frac{u}{r^m}v_{tx}\\
H_9=&&(r^mu)_{xx}\Bigg\{\frac{(r^mu)_{xx}}{r^{2m}}-4m\frac{(r^mu)_xv}{r^{3m+1}}-2m\frac{uv_x}{r^{2m+1}}+2m(3m+1)\frac{uv^2}{r^{3m+2}}\Bigg\}-m\frac{u^2}{r^{m+1}}v_{xxx},\\
&&+\gamma\frac{v_x}{v^{\gamma+1}}v_{xxx}-\frac{2m\alpha\tilde{\mu}}{v^{\alpha+1}r}uv_xv_{xxx}+\frac{2\tilde{\mu}+\tilde\lambda}{v^{\alpha+1}}(r^mu)_{xx}v_{xxx}-\frac{2\tilde\mu}{v^2}+\frac{(2\tilde{\mu}+\tilde\lambda)(\alpha+1)}{v^{\alpha+2}}v_x(r^mu)_xv_{xxx}\\
&&-v_{xxx}\Bigg\{4m\frac{r^{m-1}}{v^{\beta+4}}v_{xx}-2(\beta+5)\frac{r^{2m}}{v^{\beta+6}}v_xv_{xx}+2m(m-1)\frac{v_x}{v^{\beta+3}r^2}-3m(\beta+4)\frac{r^{m-1}}{v^{\beta+5}}v_x^2\\
&&~~~~~~~~~~~+\frac{(\beta+5)(\beta+6)}{2}\frac{r^{2m}}{v^{\beta+7}}v_x^3\Bigg\}.
\end{eqnarray*}
Integrating  $(\ref{3.33})$ w.r.t. $t$ and $x$  over $[0,T]\times\mathbb{R}^+$, and noticing that $v_{tx}(t,0)=0$ for all $t\in[0,T]$, and applying  the Cauchy inequality, we obtain
\begin{eqnarray}\label{3.34}
&&\int_0^t\int_{\mathbb{R}^+}\frac{r^{2m}}{v^{\beta+5}}v_{xxx}^2dxd\tau\notag\\
\leq&&C_{32}\left(\|(u,r^mv_{xx},\frac{(r^mu)_x}{r^m})(t)\|^2+\|(u_0,r^mv_{0xx},\frac{(r^mu)_{0x}}{r^m})\|^2\right)+\int_0^t\int_{\mathbb{R}^+}H_9dxd\tau.
\end{eqnarray}
Furthermore, by applying the Sobolev inequality, Cauchy inequality, Young inequality, Lemma \ref{L3.6} and Corollary \ref{C3.1}, we have
\begin{eqnarray}\label{3.35}
\begin{split}
&\int_0^t\int_{\mathbb{R}^+}H_4dxd\tau\notag\\
\leq&\frac{1}{2}\int_0^t\int_{\mathbb{R}^+}\frac{r^{2m}}{v^{\beta+5}}v_{xxx}^2dxd\tau+C_{33}\int_0^t\left\|\left(v_x,r^{2m}v_{xx},(r^mu)_x,(r^mu)_{xx}\right)(\tau)\right\|^2d\tau+C_{34}(T).
\end{split}
\end{eqnarray}
Here $C_{32}, C_{33}$ are  positive constants depending only on $\alpha, \beta , \gamma, \underline{V}, \overline{V}, \|v_0-1\|_1, \|u_0\|, \|\frac{(r^mu_0)_x}{r^m}\|, \|r^mv_{0x}\|$ and $\|r^mv_{0xx}\|$.  Then \eqref{3.32} follows immediately by combining (\ref{3.34})-(\ref{3.35}),  and using Corollary $\ref{C3.1}$, Lemma \ref{L3.6} and the smallness of $\varepsilon$ such that $(C_{32}+C_{33})\varepsilon<\frac{1}{4}$. This completes the proof of Lemma \ref{L3.7}.

The same as above section, Iit follows from Corollary \ref{C3.1} and Lemmas \ref{L3.5}--\ref{L3.7} we get
\begin{align}\label{3.36}
    &\|v(t)-1\|_1^2 +\|u(t)\|_1^2+ \|v_x(t)\|_{1,r}^2+\int_0^t \left( \|v_x(\tau)\|^2 + \|v_{xx}(\tau)\|_{1,r}^2+\|u_x(\tau)\|_{1,r}^2\right) d\tau \\
    \leq& C_{35}(T)\left(\|v_0-1\|_1^2 +\|u_{0}\|_1^2+ \|v_{0x}\|_{1,r}^2+1\right), \notag
\end{align}
where $C_{35}(T)$ is a positive constant depending on $T, \alpha, \beta, \gamma, \underline{V}, \overline{V}, \|v_0-1\|_1, \|u_{0}\|_1$ and $\|v_{0xx}\|_{1,r}$. Proposition \ref{P3.1} follows immediately from (\ref{3.36}).

\paragraph{Proof of Theorem \ref{T1.2}.}
~~~~With the a priori estimates at hand, the global well-posedness result follows from a standard continuation argument. More precisely, one first constructs a solution by applying a contraction mapping principle. The obtained a priori bounds then guarantee that this local solution can be extended step by step beyond any finite time interval, which yields the global existence. This completes the proof of Theorem \ref{T1.2}. \qed
\vspace{3cm}
\begin{center}
\section*{Acknowledgements}
\end{center}
The authors would like to thank the anonymous referees for many helpful suggestions which improved the presentation of this paper. The authors are also grateful to the authors of the references cited in this work, whose valuable contributions have provided important foundations and inspiration for the present research.

\vspace{0.8cm}

\section*{ \begin{center} Declarations
\end{center}}

\subsection*{Funding}
This work was supported by the National Natural Science Foundation of China (No. 12171001),
the Support Program for Outstanding Young Talents in Universities of Anhui Province (No. gxyqZD2022007),
and the Excellent University Research and Innovation Team in Anhui Province (No. 2024AH010002).

\subsection*{Author Contributions}
 Material preparation, analysis and writing of the manuscript were performed by Fanfan Jiang.
Supervision and revision were completed by Prof. Chen.
All authors have read and approved the final manuscript.

\subsection*{Data Availability Statement}
Data sharing is not applicable to this article as no datasets were generated or analyzed during the current study.

\subsection*{Conflict of Interest}
The authors declare that they have no conflict of interest.

\subsection*{Clinical Trial Number}
Clinical trial number: not applicable.

\end{document}